\documentclass[10pt,twoside]{amsart}
\usepackage{charter, amsmath,amsbsy,amsfonts,amssymb,
stmaryrd,amsthm,mathrsfs,graphicx, amscd, tikz-cd, cancel, array, tabulary}
\usepackage{srcltx}
\usepackage[mathscr]{eucal}

\usepackage[
	hypertexnames=false,
	hyperindex,
	pagebackref,
	pdftex,
	breaklinks=true,
	bookmarks=false,
	colorlinks,
	linkcolor=blue,
	citecolor=red,
	urlcolor=red,
]{hyperref}

\def\CC{{\mathbb C}}

\def\HH{{\mathbb H}}

\def\PP{{\mathbb P}}
\def\QQ{{\mathbb Q}} 
\def\RR{{\mathbb R}}

\def\XX{{\mathbb X}} 
\def\ZZ{{\mathbb Z}} 

\def\g{{\gamma}} 

\newcommand\ssm{\smallsetminus}

\def\G{{\Gamma}}
\def\Acal{{\mathcal A}}

\def\Ccal{{\mathcal C}}
\def\Cfrak{{\mathfrak C}}

\def\Ccal{{\mathcal C}}

\def\Gcal{{\mathcal G}}

\def\Hcal{{\mathcal H}}

\def\Iscr{{\mathscr I}}

\def\Jscr{{\mathscr J}}

\def\Mcal{{\mathcal M}} 
\def\Mcalbar{\overline{\Mcal}}

\def\Pcal{{\mathcal P}}
\def\Pscr{{\mathscr P}}

\def\Scal{{\mathcal S}} 
\def\Sfrak{{\mathfrak S}}

\def\Vcal{{\mathcal V}}

\def\Tscr{{\mathscr T}}

\def\pfrak{{\mathfrak p}}

\def\qfrak{{\mathfrak q}}

\def\Tfrak{{\mathfrak T}}
\def\Ucal{{\mathcal U}}
\def\Ufrak{{\mathfrak U}}
\def\ufrak{{\mathfrak u}}
\def\Vfrak{{\mathfrak V}}
\def\Wfrak{{\mathfrak W}}

\def\gfrak{{\mathfrak g}}

\def\bb{{\mathit bb}}

\def\vor{{\rm Vor}}

\def\la{\langle}
\def\ra{\rangle}

\def\perf{{\rm perf}}
\def\bs{\backslash}
\def\half{\frac{1}{2}}
\def\pt{{\scriptscriptstyle\bullet}}

\newcommand\aut{\operatorname{Aut}}
\newcommand\diff{\operatorname{Diff}}
\newcommand\hocolim{\operatorname{hocolim}}
\newcommand\hyp{\mathcal{\Hcal}\!\mathit{yp}} 
 
\newcommand\inn{\operatorname{In}}

\newcommand\re{\operatorname{Re}}
\newcommand\spec{\operatorname{Spec}}

\newcommand\Lie{\operatorname{Lie}}
\newcommand\GL{\operatorname{GL}}

\newcommand\GLfrak{\operatorname{\mathfrak{gl}}}
\newcommand\SL{\operatorname{SL}}

\newcommand\Sp{\operatorname{Sp}}
\newcommand\symp{\operatorname{\mathfrak{sp}}}

\newcommand\Spcal{\operatorname{{\mathcal S}\!{\mathit p}}}

\newcommand\Star{\operatorname{Star}}
\newcommand\sym{\operatorname{Sym}}
\newcommand\Hom{\operatorname{Hom}}

\newcommand\Sq{\operatorname{Sq}}
\newcommand\rad{\operatorname{rad}}
\newcommand\rk{\operatorname{rk}}

\newtheorem{theorem}{Theorem}[section] 
\newtheorem{lemma}[theorem]{Lemma}

\newtheorem{proposition}[theorem]{Proposition}
\newtheorem{scholium}[theorem]{Scholium}

\theoremstyle{definition} \newtheorem{definition}[theorem]{Definition}
\newtheorem{example}[theorem]{Example}

\theoremstyle{remark} \newtheorem{remark}[theorem]{Remark}

\begin{document}


\title{The homotopy type of the Baily-Borel and allied compactifications}

\author{Jiaming Chen}
\address{Universit\'e Pierre et Marie Curie, 4 Place Jussieu, 75005 Paris (France).} 
\email{jiaming.chenl@gmail.com}
\author{Eduard Looijenga}
\address{Yau Mathematical Sciences Center, Tsinghua University, Beijing (China) and Mathematisch Instituut Universiteit, Utrecht (Nederland)}
\email{eduard@math.tsinghua.edu.cn}

\begin{abstract}
A number of compactifications familiar in complex-analytic geometry, in particular the Baily-Borel compactification and its toroidal variants,  as well as the Deligne-Mumford compactifications, can be covered by open subsets whose nonempty intersections are $K(\pi,1)$'s. We exploit this fact to  define a `stacky homotopy type' for these spaces as the homotopy type of  a small category. We thus  generalize an old result of Charney-Lee on the Baily-Borel compactification of $\Acal_g$ and  recover (and rephrase) a more recent one of Ebert-Giansiracusa on the Deligne-Mumford compactifications. We also describe an extension of the period map for Riemann surfaces (going from the Deligne-Mumford compactification to the Baily-Borel compactification of the moduli space of principally polarized varieties)  in these terms.
\end{abstract}

\maketitle

\section*{Introduction}
In a remarkable, but seemingly little noticed  paper \cite{charneylee:satake} Charney and Lee described a rational homology equivalence between 
the Satake-Baily-Borel compactification of the moduli space of principally polarized abelian varieties $\Acal_g$,  denoted here by 
$\Acal_g^\bb$, and the classifying space of a certain category which has its origin in Hermitian $K$-theory. They exploited this to show that if we let 
$g\to \infty$, the homotopy type of this classifying space (after applying the `plus construction') stabilizes and they computed its stable rational cohomology.  

Our aim was twofold: first, to put the results of that paper in a transparent  framework that lends itself to generalization, and second, to make a clearer link with  algebraic geometry. During our efforts we found that we could obtain the stable rational cohomology of the spaces $\Acal_g^\bb$  by means of relatively  conventional methods in algebraic geometry, leading us even to a determination of the mixed Hodge types of the stable classes. As this involved no category theory and hardly any homotopy theory,  we decided to put this in a separate paper \cite{chen-looijenga}. By contrast, the focus of the present article is on homotopy types and may be regarded as our proposal for accomplishing the first goal. 

Baily-Borel compactifications and Deligne-Mumford compactifications have in common that they can be obtained as orbit spaces of stratified spaces
with respect to an action of a discrete group.  But usually the stratification is not locally finite, the space not locally compact  and the action of the group  not proper, and yet,  these drawbacks somewhat miraculously cancel each other out when we pass to the  orbit space, which after all, is a compact Hausdorff space. But a feature that 
they have in common is that the strata are contractible. This  leads (in a not quite trivial manner) to an open covering of the orbit space that is 
closed under finite intersections and whose members  are Eilenberg-MacLane spaces. One of the main results of the  paper (Theorem \ref{thm:emcov2}) formalizes the type of input under which such a structure is present and then yields as output  (what we have called) the \emph{stacky}  homotopy type of the orbit space as one of the classifying space of a category. (This terminology may be somewhat misleading as we have not been able to define a stack of which this is the homotopy type, see Remark \ref{rem:stacky}  for discussion.) Our set-up is reminiscent of---and indeed, inspired by---the construction of an \'etale homotopy type.  We illustrate its efficiency by showing 
how we thus recover with little additional effort a theorem of Ebert and Giansiracusa  on the homotopy type of a Deligne-Mumford compactification as a 
stack (theirs in turn generalized another  theorem of Charney and Lee). 

Another application, and one that is more central to this paper, concerns an arbitrary  Baily-Borel compactification,
and yields a stacky homotopy type for such a space. The proof that the hypotheses of   Theorem \ref{thm:emcov2} are then  satisfied requires a good understanding of the topology of  the Satake extension of a bounded symmetric domain `with a $\QQ$-structure'. Although all we need is in a sense known, it is not so easy to winnow the relevant facts from the literature and so we have tried to present this as a geometric narrative, avoiding  any mention of root data (despite Mumford calling these in \cite{amrt} `the name of the game'). Unlike Charney and Lee we do not make use of  Borel-Serre's compactification `with corners'.  We then combine our results for $\overline\Mcal_g$ and  $\Acal^\bb_g$ to show how the stacky homotopy type of the period map extension $\overline\Mcal_g\to\Acal^\bb_g$ can be given by the classifying space construction applied to a functor. 

Finally we show that Theorem \ref{thm:emcov2} also applies to the toroidal compactifications of Ash, Mumford, Rapoport and Tai and we illustrate this with the perfect cone compactification of $\Acal_g$. \\

\emph{Notational conventions.} If a group $\G$ acts on a set $X$, then for $A\subset X$, $\G_A$ resp.\  $Z_{\G}(A)$ denotes the group $\g\in\G$ that leave $A$ invariant resp.\ fix $A$ pointwise and $\G (A)$ will stand for the quotient $\G_A/Z_{\G}(A)$. 

As a rule an algebraic group (defined over a field contained in $\RR$, usually $\QQ$) is denoted by a script capital, its Lie group of real points  by the corresponding roman capital and the Lie algebra of the latter by the corresponding Fraktur lower case. 
\\

\begin{small}
\emph{Acknowledgements.} 
We thank Sam Grushevski, Klaus Hulek and Osola Tommasi for arousing our interest  in and for correspondence relating to some of the  issues discussed here. We are also grateful to Kirsten Wickelgren and Andr\'e Henriques for helpful remarks.
\end{small}

\section{Grothendieck-Leray coverings}

Recall that every small category $\Cfrak$ defines a simplicial set  $B\Cfrak$ and hence a semi-simplicial complex (its geometric realization)
$|B\Cfrak|$.  An \emph{$n$-simplex} of $B\Cfrak$ is represented by a chain $C_0\to C_1\to \cdots \to C_n$ of $n$ morphisms in $\Cfrak$, the  
\emph{$i$th degeneracy map} produces the $(n+1)$-simplex obtained by inserting the identity of $C_i$ at the obvious place and the $i$th face map is the $(n-1)$-simplex 
obtained by omitting $C_i$ (when $i=0,n$) or replacing $C_{i-1}\to C_i\to C_{i+1}$ by the composite $C_{i-1}\to C_{i+1}$ (when $0<i<n$). Its 
geometric realization $|B\Cfrak|$ is obtained as follows. Take for every $n$-simplex 
$C_0\to C_1\to \cdots \to C_n$ as above a copy of the standard $n$-simplex $\Delta^n$ and use the face maps to make the obvious identifications 
among these copies. The resulting space has almost the structure  of a simplicial complex with each edge labeled by a $\Cfrak$-morphism (it is just that 
a simplex is in general not determined by its vertex set). We subsequently use the degeneracy maps to make further identifications:  simplices having 
all their edges  labeled by  the identity of an object of $\Cfrak$  are contracted so that in the end there is no $1$-simplex with identity label left.

For example, if we regard  a discrete group $G$ as a category with just one object  and $G$ as its set of morphisms, then this construction  reproduces a model for the classifying space of $G$. That is why we call $|B\Cfrak|$ the \emph{classifying space} of $\Cfrak$. The \emph{homotopy type of $B\Cfrak$} will mean the homotopy type of $|B\Cfrak|$. Note that  for every object $C$ of $\Cfrak$ we have a copy of $B\!\aut(C)$ in $B\Cfrak$.  
A functor $F:\Cfrak\to \Cfrak'$ induces a map $BF:B\Cfrak\to B\Cfrak'$ and a natural transformation $F_0\Rightarrow F_1$  between two such functors determines a homotopy  between the associated maps $|BF_0|$ and $|BF_1|$. In particular, an equivalence of categories induces a homotopy equivalence. 
\\

Let  $Y$ be a locally contractible paracompact Hausdorff space. Assume $Y$ endowed with an indexed open covering $\Vfrak=(V_\alpha)_{\alpha\in A}$ that is locally 
finite and closed under finite  nonempty intersection: if  $V_\alpha, V_\beta\in\Vfrak$, then $V_\alpha\cap V_\beta=V_\gamma$ for some 
$\gamma\in A$, when nonempty. These indexed open subsets  define a category  $\Vfrak$ with object set $A$ for which we have a (unique) morphism 
$\alpha\to \beta$ when $V_\alpha\subset V_\beta$.  Any partition of unity subordinate to the maximal members of $\Vfrak$ can be used to define 
a continuous map $Y\to |B\Vfrak|$. As Weil showed, this is a homotopy equivalence  when  each $V_\alpha$ is contractible.
 
Suppose now  that  every $V_\alpha$ is a $K(\pi, 1)$ instead. More specifically, assume that for every $V_\alpha$  we are given a covering map 
$U_\alpha\to V_\alpha$  with $U_\alpha$  contractible. Then we have a category $\Ufrak$ 
with again $A$ as object set, but for which a morphism  is simply a continuous map $U_\alpha\to U_\beta $ which commutes with projections onto $Y$ 
(so that then $V_\alpha \subset V_\beta$). We have an obvious functor $\Ufrak\to \Vfrak$. Notice that  for any $\alpha\in A$, $\aut_{ \Ufrak}(\alpha)$ 
is the group of covering transformations of $U_\alpha\to V_\alpha$ and hence is  isomorphic to the fundamental group of $U_\alpha$.  
This means that $|B\!\aut_{\Ufrak}(\alpha)|$ is homotopy equivalent to $U_\alpha$.
The following theorem is mentioned  by Sullivan as Example 3 on page 125 of \cite{sullivan})  who refers in turn to  Theorem 2 on  p. 475 of 
Lubkin' s paper \cite{lubkin} (we thank Kirsten Wickelgren for pointing out these references).

\begin{theorem}[Lubkin, Sullivan]\label{thm:emcov}
In this situation the  continuous map  $Y\to |B \Vfrak|$ defined by a partition of unity  lifts to 
$Y\to |B(\Ufrak)|$ and this lift is a homotopy equivalence.
\end{theorem}

For the applications that we have in mind  we need a generalization of this theorem of a `stacky' nature.  
To be precise, we assume  that  $U_\alpha$  is still contractible, but that  we are now given a group $\G_\alpha$ acting properly discontinuously  on $U_\alpha$ with a subgroup of finite index acting freely, such that $\pi_\alpha: U_\alpha\to V_\alpha$ is the formation of the $\G_\alpha$-orbit  space. Note that  $V_\alpha$ is then paracompact Hausdorff.

Let $V_\alpha\subset V_\beta$.  Let us agree that an \emph{admissible lift} of the inclusion $V_\alpha\subset V_\beta$  is a pair 
$(j: U_\alpha\to U_\beta,\phi: \G_\alpha\to \G_\beta)$ for  which 
\begin{enumerate}
\item [(AL$_1$)] $\phi: \G_\alpha\to \G_\beta$ is a group homomorphism,
\item [(AL$_2$)] $j$ lifts the inclusion  $V_\alpha\subset V_\beta$ and is equivariant relative to $\phi$, and
\item [(AL$_3$)] $\phi$ maps the $\G_\alpha$-stabilizer of every $x\in U_\alpha$ onto the
$\G_\beta$-stabilizer of $j(x)$.
\end{enumerate}
The group $\G_\beta$ also acts on the admissible lifts of $V_\alpha\subset V_\beta$ by having $\g\in\G_\beta$ send $(j, \phi)$ to $(\g j, \inn(\g)\phi)$,  where $\inn(\g)$ is the inner automorphism of  $\G_\beta$ defined by $\g$. We observe that if $\G_\beta$ acts freely on a  connected open-dense subset of the preimage of $V_\alpha$ in $U_\beta$, then this action is simply transitive.  

\begin{definition}\label{def:gl}
A \emph{Grothendieck-Leray atlas} $\Ucal$ over $Y$ consists of a collection of pairs $(\G_\alpha, \pi_\alpha: U_\alpha\to V_\alpha )_{\alpha\in A}$  as above and assigns to
every inclusion $V_\alpha\subset V_\beta$ a $\G_\beta$-orbit of admissible lifts  $(j,\phi)$, such that these are the morphisms of  a category  $\Ucal$: the identity of the pair $(U_\alpha, \G_\alpha)$ defines an admissible lift  and the composite of two admissible lifts is again admissible. 

A \emph{principal Grothendieck-Leray atlas} $\Ufrak$ over $Y$  is a Grothendieck-Leray atlas for which these lifts are indexed in a particular way: it 
consists  of giving for every inclusion $V_\alpha\subset V_\beta$ a collection of 
admissible lifts  indexed by a \emph{principal} $\G_\beta$-set $I_\alpha^\beta$: $\Phi_\alpha^\beta=(j_{i},\phi_{i})_{i\in I_\alpha^\beta} $ together with maps
$\Phi_\beta^\g\times \Phi_\alpha^\beta\to \Phi_\alpha^\g$ defined whenever $V_\alpha\subset V_\beta\subset V_\g$ such that
\begin{enumerate}
\item [(GL$_1$)] we have $I_\alpha^\alpha=\G_\alpha$ with $1\in \G_\alpha$ defining the pair $(1_{U_\alpha}, 1_{\G_\alpha})$, 
\item [(GL$_2$)] for $i\in I_\alpha^\beta$ and  $g\in\G_\beta$ we have $j_{g(i)}^\beta=gj_{i}$ and $\phi_{g(i)}^\beta=\inn (g)\phi_{i}$ and
\item [(GL$_3$)] the map $\Phi_\beta^\g\times \Phi_\alpha^\beta\to \Phi_\alpha^\g$ is $\G_\g$-equivariant and defines the composition of 
 admissible lifts.
\end{enumerate}
We often regard $\Ufrak$ as a small category with object set $A$ such that $\Phi_\alpha^\beta$ is the set of morphisms $\alpha\to \beta$. 
\end{definition}

\begin{remark}\label{rem:glprincipal}
A  Grothendieck-Leray atlas is automatically principal if each $\G_\alpha$ acts faithfully on $U_\alpha$, for then the collection of all the lifts $U_\alpha\to U_\beta$ of $V_\alpha\subset V_\beta$ are simply transitively permuted by $\G_\beta$ and hence form a principal $\G_\alpha$-set. 
\end{remark}

\begin{remark}\label{rem:stacky}
A  Grothendieck-Leray atlas gives rise to a Deligne-Mumford stack if its admissible lifts have the property that in (AL$_3$) 
$\phi$ maps the $\G_\alpha$-stabilizer of every $x\in U_\alpha$ \emph{isomorphically} onto the $\G_\beta$-stabilizer of $j(x)$. Although the structure that we get in general is weaker, there is a notion of a \emph{local chart}: given $y\in Y$, then the $V_\alpha$'s containing $y$ are finite in number and their intersection is one of them, say $V_{\alpha_o}$. We then stipulate that  for every $x\in \pi_{\alpha_o}{}^{-1}(y)$,  the pair $(U_{\alpha_o}\to Y, x) $ defines a local chart.
If $\alpha\in A$ is such that  $y\in V_\alpha$, then there exists by definition an admissible lift $(j,\phi)$ of the inclusion $V_{\alpha_o}\subset V_\alpha$ and $\phi$ maps  the $\G_{\alpha_o}$-stabilizer of $x$  onto the $\G_\alpha$-stabilizer of $j(x)$. If this is in fact an isomorphism,  
then we declare that the pair  $(U_\alpha\to X, j(x))$ is also a local chart. But the property of being a local chart need be not open: there exist examples for which the set of $x'\in U_{\alpha_o}$ for which $(U_{\alpha_o}\to Y, x')$  is a chart fails to be a neighborhood of $x$. All we can say a priori is that $(U_{\alpha_o}\to Y, x')$ is a local chart when $\pi_{\alpha_o}(x')$ lies in $V_{\alpha_o}\ssm \cup_{y\notin V_\beta} V_\beta$. This is a closed subset of $V_{\alpha_o}$ which contains $y$ and so this only shows that we have a locally finite partition of $Y$ into locally closed subsets along which charts `propagate'. This  phenomenon we encounter for a Baily-Borel compactifications, where the partition is that into Baily-Borel strata.
\end{remark}

We associate to a Grothendieck-Leray atlas as above a homotopy type that we will refer to as its \emph{stacky homotopy type}.
Let us begin with recalling Segal's categorical construction  of the universal bundle of a discrete group $\G$ \cite{segal}.
Let  $\hat\G$ be the groupoid  whose object set is $\G$  and has for any two objects $\g, \g'\in\G$ just one morphism $\g\to \g'$. Since this 
category is equivalent to the  subcategory represented by the single element $1\in\G$, $|B\hat\G|$ is contractible.  This category is acted on 
by the group $\G$ with  quotient category the group $\G$, but now viewed as a category with a single object: the quotient forming  functor $\hat\G\to \G$ sends the unique morphism  $\g\to \g'$ to $\g^{-1}\g'$. The associated  map $|B\hat\G|\to |B\G|$ is a universal $\G$-bundle. 
This construction is clearly functorial on the category of discrete groups.

We apply this in the present situation as follows. For $\alpha\in A$, $\hat U_\alpha:= U_\alpha\times |B\hat\G_\alpha|$ is contractible and the diagonal action of $\G_\alpha$ on it is free and proper.  So if we denote by $\hat U_\alpha\to \hat V_\alpha$ the formation of the corresponding orbit space, then this is also a universal $\G_\alpha$-bundle.  Given an inclusion 
$V_\alpha\subset V_\beta$, then an admissible lift  $(j:U_\alpha\to U_\beta, \phi:\G_\alpha\to \G_\beta)$ defines a map $\hat j:=j\times |B\hat\phi|: \hat U_\alpha\to \hat U_\beta$
that is equivariant with respect to $\phi$. Such lifts make up a single $\G_\beta$-orbit and hence we have a  map between two universal coverings:
they induce the same map $\hat V_\alpha\to \hat V_\beta$ and they yield all the lifts $\hat U_\alpha\to \hat U_\beta$ of the latter. Our assumptions imply that $\alpha\mapsto \hat V_\alpha$ defines a functor from $\Vfrak$ to the category of topological spaces so that we can form
$\hat Y:=\varinjlim_{\Vfrak} \{\hat V_\alpha\}_\alpha$.  The collection of  the 
maps $\hat U_\alpha\to \hat V_\alpha$ plus the lifts $\hat j$ as above form a category $\hat\Ufrak$ of contractible spaces over $\hat Y$. The Lubkin-Sullivan theorem does not quite apply as such to this system of coverings, because the maps
$\hat V_\alpha\to \hat V_\beta$ need not be injective (they are open, though). But it will, if we replace $\hat Y$ by the homotopy colimit $\hat Y^h:=\hocolim_{\Vfrak}  \hat V_\alpha$ of this system (here we use the construction that regards the system as a simplicial space). It has the property that the natural map $\hat Y^h\to \hat Y$ is a homotopy equivalence. We thus find a homotopy equivalence between $\hat Y$ and $|B\hat\Ufrak|$. 

Consider the obvious projection $p_\alpha: \hat V_\alpha\to V_\alpha$. The fiber over $y\in V_\alpha$ is the quotient of the contractible $\G_\alpha$-space  $|B\hat\G_\alpha|$ by the $\G_\alpha$-stabilizer of  some  $x\in U_\alpha$ over $y$. So it  has the rational cohomology of the  finite group $(\G_\alpha)_x$, which is that of a point. This fiber is also a deformation retract of the preimage of a neighborhood of $y$ in $V_\alpha$. Hence the Leray spectral sequence 
for rational cohomology of the projection $\hat Y\to Y$ degenerates so that this projection induces an isomorphism on rational cohomology. 

In case we have a principal Grothendieck-Leray atlas $\Ufrak$, then we can identify $\G_\alpha$ with the $\Ufrak$-endomorphisms of $\alpha$ so that 
$B\G_\alpha\subset B\Ufrak$.
The projection $\hat V_\alpha\to |B\G_\alpha|$ is a bundle with fiber the contractible $U_\alpha$. Since this is functorial, these projections assemble  to a map $\hat Y^h\to |B\Ufrak|$. Its fibers are contractible and so this is a homotopy equivalence. 

We record this discussion in the form of a scholium.

\begin{scholium}\label{thm:vemcov}
With a Grothendieck-Leray atlas as above we have associated a natural homotopy  class of maps  from its stacky homotopy type to $Y$ and this class induces an isomorphism on rational cohomology. For a principal Grothendieck-Leray atlas $\Ufrak$  this stacky homotopy type is represented by $|B\Ufrak|$.
\end{scholium}

\begin{remark}\label{rem:refininggl}
In our applications we encounter refinements of Grothendieck-Leray atlases  of very simple type, namely obtained  by giving for each $\alpha\in A$ an open $V_\alpha'\subset V_\alpha$ such that this inclusion is a homotopy equivalence and $\{V'_\alpha\}_\alpha$ still covers $Y$. This extends in a natural manner to a Grothendieck-Leray atlas with the same index set and if one is principal, then so is the other. It is clear that this induces a homotopy equivalence between the associated homotopy types. From a conceptual point of view it  would be more satisfying to introduce a considerable more general notion of refinement for Grothendieck-Leray atlases: such a refinement  should then be given  by a functor $F:\Ucal\to \Ucal'$ that gives rise to a (weak) homotopy equivalence of their stacky homotopy types so that the resulting structure on $Y$ (which we might regard as a weak form of a Deligne-Mumford stack) has this (weak)  homotopy type as one its attributes. We refrained from developing these notions,  as there is for this no need in the present paper. 
\end{remark}

Our applications of this theorem have in common a number of features that are worth isolating. 
Let $X$ be a space endowed with a \emph{stratification} $\Scal$, that is, a partition  into subspaces (called strata) such that the closure  of each stratum is a union of strata. We then have a partial order on $\Scal$ for which $S'\le S$ means that $S'\subset\overline{S}$.
We assume that the length of chains $S_\pt=(S_0>S_1>\cdots >S_n)$ in $\Scal$ is bounded, but we do not ask that $X$ be locally compact, nor that
$\Scal$ be locally finite.

\begin{theorem}\label{thm:emcov2}
Let $\G$ be a discrete group which acts on the stratified space $(X,\Scal)$  and suppose that for every stratum $S$ we are given a (what we will call 
\emph{link-}) subgroup
$\G^\ell_S\subset Z_\G (S)$ such  that for all $\g\in\G$, $\G^\ell_{\g S}=\g \G^\ell_S\g^{-1}$ (so that $\G^\ell_S$ is normal in 
$\G_S$) and is such that  $\G^\ell_S\supset \G^\ell_{S'}$ when  $S\le S'$.  

If we can find for every $S\in\Scal$ an open neighborhood $U_S$ of $S$ in $X$ such that 
\begin{enumerate}
\item[(i)] $U_S\cap U_{S'}$ is empty unless $S'\ge S$ or $S'\le S$,
\item[(ii)] $\g (U_S)=U_{\g S}$ for every $\g\in\G$,
\item[(iii)]  for every stratum $S$, $\G^\ell_S\bs U_S$ is a paracompact Hausdorff space on which $\G_S/\G^\ell_S$ acts properly with a cofinite subgroup acting freely,
\item[(iv)] for every chain $S_\pt=(S_0>S_1>\cdots >S_n)$ of strata,   
$\G^\ell_{S_0}\bs (U_{S_0}\cap \cdots\cap U_{S_n})$ is contractible,
\end{enumerate}
then the orbit space $\G\bs X$ is a paracompact Hausdorff space which comes with a  natural structure of a stacky homotopy type (so independent of
the choice of open subsets $U_S$ as above) that is represented by the category $\Sfrak$  with  object set $\Scal$  and for which a morphism $S \to S'$ is a right coset 
$[\gamma]\in \G^\ell_{S'}\bs\G$ with the property that  $\gamma S \ge S'$ (so that we have natural homotopy class of maps $|B\Sfrak|\to \G\bs X$ which induces an isomorphism on rational cohomology). This is functorial with respect to  inclusions $X'\subset X$ of open $\G$-invariant 
unions of strata.

If in this situation $\G$ acts faithfully and the action in (iii) is free (so that necessarily $\G^\ell_S= Z_\G (S)$ for  every $S\in\Scal$), then in the preceding `stacky homotopy' can be replaced by `homotopy'.
\end{theorem}
\begin{proof}
We note that (i) implies that any finite nonempty intersection of such $U_S$ is of the form
$U_{S_\pt}=U_{S_0}\cap \cdots \cap U_{S_n}$ for a unique chain $S_\pt=(S_0>S_1>\cdots >S_n)$ in $\Scal$. From (i) and (ii) we get that every $\G$-orbit meets $U_S$ in a $\G_S$-orbit or is empty. Hence $\G_S\bs U_S$ maps homeomorphically onto an open subset  $V_S$ of $\G\bs X$.
Any nonempty intersection of such open subsets of $\G\bs X$ is the image $V_{S_\pt}$ of $U_{S_\pt}:=U_{S_0}\cap \cdots\cap U_{S_n}$ for some chain $S_\pt$ and hence homeomorphic to $\G_{S_\pt}\bs U_{S_\pt}$. If we put $\overline{U}_{S_\pt}:=\G^\ell_{S_0}\bs U_{S_\pt}$, then $\overline{U}_{S_\pt}$ is an open subset of $\overline{U}_{S_0}=\G^\ell_{S_0}\bs U_{S_0}$. By (iii) and (iv)  this is a contractible paracompact Hausdorff  space
on which $\overline\G_{S_\pt}:=\G_{S_\pt}/\G^\ell_{S_0}$ acts properly.  

We claim that the collection of pairs $(\overline{U}_{S_\pt}, \overline\G_{S_\pt})$ extends in a natural manner to  a principal Grothendieck-Leray atlas: 
let $S_\pt$ and $S'_\pt$ be finite chains in $\Scal$ such that the image of $\overline{U}_{S_\pt}$  in $\G\bs X$ is contained in the image of $\overline{U}_{S'_\pt}$.
This is equivalent to  the existence of a $\g\in\G$ such that $S'_\pt$ is a subchain of $\g S_\pt$ and the elements of  $\G$ with this property then make up the right coset $\G_{S'_\pt}\g$. The smaller coset  $\G^\ell_{S'_0}\g$ defines an admissible lift: since $\g\G^\ell_{S_0}=\G^\ell_{\g S_0}\g\subset \G^\ell_{S'_0}\g$, this indeed induces a continuous map $j: \overline{U}_{S_\pt} \to \overline{U}_{S'_\pt}$ over $\G\bs X$
and since $\g\G_{\Sfrak_\pt}\g^{-1}=\G_{\g\Sfrak_\pt}\subset  \G_{\Sfrak'_\pt}$,  conjugation by $\g$ defines a homomorphism $\phi:=\overline\G_{\Sfrak_\pt}\to \overline \G_{\Sfrak'_\pt}$ such that $j$ is $\phi$-equivariant.  So we have a  collection of admissible lifts
indexed by the $\G^\ell_{\g S'_0}$-cosets contained in $\G_{S'_\pt}\g$. This is clearly a principal set for the group $\overline\G_{S'_\pt}= \G_{S'_\pt}/ \G^\ell_{\g S'_0}$. The other three properties  of Definition \ref{def:gl} are now easily checked.

So the associated category $\Sfrak_\pt$ has as its objects the finite chains in $\Scal$ and a morphism 
$S_\pt\to S'_\pt$ is given by right coset $[\g]\in \G^\ell_{S'_0}\bs\G$ such $S'_\pt$ is a subchain of $\g S_\pt$. Strictly speaking we do not have principal Grothendieck-Leray atlas yet, because of an `overcount' in our indexing: the image $V_{S_\pt}$ of $\overline{U}_{S_\pt}$ in $\G\bs Y$ is of course also the image of $\g\overline{U}_{S_\pt}$ and in this way we get $\#(\G/\G_{S_\pt})$ copies of $\overline{U}_{S_\pt}$ having the same image. So in this rather trivial sense the cover 
$\{V_{S_\pt}\}$ can fail to be locally finite. But we can of course select for each $\G$-orbit of  $\Sfrak_\pt$-objects a representative and then take the full subcategory $\Sfrak^\circ_\pt\subset \Sfrak_\pt$ with this collection of objects. We then get a principal Grothendieck-Leray atlas and since 
$\Sfrak^\circ_\pt\subset \Sfrak_\pt$ is an equivalence of categories, the stacky homotopy type of $\G\bs Y$ is that of $|B\Sfrak_\pt|$.

We have a functor $F: \Sfrak_\pt\to \Sfrak$ defined by  $S_\pt=(S_0>S_1>\cdots >S_n)\mapsto S_0$. Indeed, a morphism 
$[\g]: S_\pt \to S'_\pt$ as above has the property that $S'_0=\g S_i$ for some $i$ and so $F(S'_\pt)=S'_0=\g S_i\le \g S_0=\g F(S_\pt)$. Since  
$\g\G^\ell_{S_0}\subset \g\G^\ell_{S_i}=\G^\ell_{S'_0}\g$, $\g $ determines an element $[\g]$  of $\G^\ell_{S'_0}\bs \G$ and this yields our $ \Sfrak$-morphism $F[\g]: S_0\to S'_0$.

According to Thm.\ A of \cite{quillen},  $|BF|$ is a homotopy equivalence if we show that for every object $S\in\Scal$ of $\Sfrak$, the category  $F/S$ is contractible. 
Let us recall that an an object of $F/S$ is given by pair $(S_\pt, [\g])$, where $S_\pt=(S_0>S_1>\cdots >S_n)$ is an object of $\Scal_\pt$ and 
$[\g] \in \G^\ell_{S_0}\bs\G$ is such that $\g S_0\ge S$. An $F/S$-morphism $(S_\pt, [\g])\to (S'_\pt, [\g'])$ is a $\Sfrak_\pt$-morphism
$[\delta]: S_\pt\to S'_\pt$ (with  $[\delta]\in \G^\ell_{S'_0}\bs\G)$,  so that $S'_\pt$ is a subchain of $\g S_\pt$ with the property that $\g'\delta$ and $\g$ define the same element of    $\G^\ell_{S}\bs \G$. This category  has  as  a  final object, namely  $(S, [1])$:  for an  object  $(S_\pt, [\g])$ of $F/S$, 
$[\g]$ defines an $F/S$-morphism $(S_\pt, [\g])\to (S,  [1])$. This implies that $F/S$ is contractible.

The last assertion  is obtained by applying Theorem \ref{thm:emcov} instead of \ref{thm:vemcov}.
\end{proof}

In many applications, we will take $\G^\ell_S=Z_\G (S)$, but this need not be so in the situation that is our main interest, the Baily-Borel compactification.
It is also with this case in mind that we  included a stacky version.

Here is perhaps the simplest nontrivial illustration of Theorem \ref{thm:emcov2}. 

\begin{example}[The infinite ramified cover of the unit disk]\label{example:easy}
We take for $X$ be the space that contains the upper half plane $\HH$ as an open subset and for which the complement 
$X-\HH$ is a singleton $\{\infty\}$.  A neighborhood basis of $\infty$  meets $\HH$ in the upwardly  shifted copies of $\HH$.
We take this partition as our stratification $\Scal$ and we take $\G=\ZZ$, with $\G$ acting by translations on $\HH$ (and of course trivially on $\infty$)
and $\G^\ell_{\{\infty\}}=Z_\G(\{\infty\})=\ZZ$ and  $\G^\ell_{\HH}=Z_\G(\HH)=\{ 0\}$.  We choose $U_{\{\infty\}}=X$ and $U_\HH=\HH$. 
The category $\Sfrak$ that we get from Theorem \ref{thm:emcov2} has the two objects $\{\infty\}, \HH$ with $\{\infty\}$ being a final object.
The only $\Sfrak$-morphisms apart from the unique morphism $\HH\to \{\infty\}$ are the elements of the (translation) group $\ZZ$ viewed  as automorphisms of $\HH$. So $|B\Sfrak|$ can be identified with the cone over the classifying space $|B\ZZ|$.

The map $z\mapsto \exp(2\pi\sqrt{-1}z)$ identifies the pair $\ZZ\bs (X, \HH)$ with the pair 
$(\Delta, \Delta^*)$ consisting of the complex unit disk and  the same deprived from $0$. So if we consider $\Delta^*$ as the primary datum, then we are just filling in the puncture and in the above picture $\Delta^*\subset \Delta$ corresponds to the inclusion of $|B\ZZ|$ in the cone over $|B\ZZ|$.
\end{example}

This example generalizes in a  simple manner to the product $(\Delta^n, (\Delta^*)^n)$ (that we obtain as an orbit  space of
$(\HH\cup\{ \infty\})^n$ under the action of $\ZZ^n$). Closely related to this is the example below of a torus embedding. It appears implicitly in some of our applications. 

\begin{example}\label{example:torus}
Let $\G$ be a free abelian group of finite rank. Then $T=\CC^\times \otimes \G$ is an algebraic torus with underlying affine variety $\spec(\CC[\G^\vee])$, where $\G^\vee=\Hom (\G, \ZZ)$. Let also be given a closed strictly convex cone $\sigma\subset \RR\otimes\G$ spanned by a finite subset of $\G$. Recall that this defines a normal affine torus embedding $T\subset T^\sigma$ as follows.
Denoting by $\check{\sigma}\subset \Hom (\G, \RR)$ the cone of linear forms that are $\ge 0$ on $\sigma$, then $T^\sigma:=\spec \CC[\G^\vee\cap \check{\sigma}]$ and the inclusion $\CC[\G^\vee]\supset \CC[\G^\vee\cap \check{\sigma}]$ defines 
the embedding $T\subset T^\sigma$. We also recall that $T^\sigma$ is stratified into algebraic tori that are quotients of $T$ and indexed by the faces of $\sigma$: 
for every face $\tau$ of $\sigma$ denote by $\G_\tau$ the intersection of $\G$ with the vector subspace of $\RR\otimes \G$ spanned by $\tau$ and put $T_\tau:= \CC^\times \otimes \G_\tau$. Then $T(\tau):=T/T_\tau$ is a stratum.

But in this context it is better  to think of $T$ (via the exponential map) as the orbit space of its Lie algebra $\mathfrak{t}=\CC\otimes\G$ by $\G$, letting each $\g\in \G$ act as  translation over $2\pi\sqrt{-1}\g$. There is then a corresponding picture for $T^\sigma$: if we write $\mathfrak{t}_\tau$ for the $\CC$-span of $\tau$, then $T^\sigma$ is the orbit space with respect to the obvious $\G$-action on the disjoint union of the complex vector spaces $\mathfrak{t}^\sigma:=\sqcup_{\tau\le\sigma} \mathfrak{t}/\mathfrak{t}_\tau$ (endowed with a topology which is defined in the spirit of Example \ref{example:easy}). We define a neighborhood $U_\tau$ of $\mathfrak{t}/\mathfrak{t}_\tau$  in $\mathfrak{t}^\sigma$ as follows: let $\Phi\subset \check{\sigma}\cap \G^\vee$ be the set of integral generators of the one-dimensional faces of $\check{\sigma}\cap\G^\vee$. Then we define $U_\tau$ as the subset of
$\sqcup_{\rho\le \tau} (\mathfrak{t}/\mathfrak{t}_\rho)$ defined by the property that its intersection with $\mathfrak{t}/\mathfrak{t}_\rho$ is defined by  $\re (\phi)>\re(\phi')$ for all $(\phi, \phi')\in \Phi\times \Phi$ with $\phi|\tau>0$ and $\phi'|\tau=0$ (note that that both $\phi$ and $\phi'$ define linear forms on $\mathfrak{t}/\mathfrak{t}_\rho$). Then we have $\G_{U_\tau}=\G$ and  $Z_\G (\tau)=\G\cap \mathfrak{t}_\tau$. Since $(\G\cap \mathfrak{t}_\tau)\bs U_\tau$ fibers over $\mathfrak{t}/\mathfrak{t}_\tau$ with fibers conical open subsets of complex vector space it is contractible.
The associated category $\Sfrak$ has its objects indexed by faces $\tau$ of $\sigma$, and a morphism $\tau\to \tau'$ only exists when $\tau\subset \tau'$ and is then given by an element of 
$\G (\tau'):=\G/\G\cap \mathfrak{t'}_\tau$. This category has a final object represented by $\tau=\sigma$ and so $|B\Sfrak|$ is contractible. We may  also  obtain $|B\Sfrak|$ as the geometric realization of the diagram of spaces  $B \G (\tau)$ connected by the  maps 
$B\G (\tau) \to B\G (\tau')$  ($\tau\subset \tau'$).
\end{example}

\section{The homotopy type of a Deligne-Mumford compactification}
Ebert and Giansiracusa  determined in \cite{ebertgian}  the homotopy type of the Deligne-Mumford moduli space of stable $n$-punctured genus $g$ curves. We outline how this fits our setting.  This is one which involves the rational homology type only, but in the present case our arguments work without change if we wish to do this for the homotopy type of that moduli space as an orbifold. 

We fix a \emph{$n$-punctured surface $S$ of genus $g$,}  which means that $S$ is a connected  oriented differentiable surface that can be obtained as the complement  of $n$ distinct points of a compact surface of genus $g$. We assume that $S$ is \emph{hyperbolic} in the sense that its Euler characteristic $2-2g-n$ is negative. This is indeed equivalent to $S$ admitting a complete metric of constant curvature $-1$ and of finite volume (and such a metric is equivalent to putting on $S$ a complex structure compatible with the given orientation so that it becomes a nonsingular complex-algebraic curve which is universally covered by the upper half plane). 
Denote by $\hyp(S)$ the space of all such metrics  on $S$. This space is acted on by the group $\diff (S)$ of diffeomorphisms of $S$. The identity component $\diff^0(S)$ of $\diff(S)$ acts freely and its orbit space,  the \emph{Teichm\"uller domain} $\Tscr (S)$ of $S$, is contractible and has naturally the  structure of a complex manifold of complex dimension $3g-3+n$.  Letting $\diff^+(S)\subset \diff (S)$ stand for the group of orientation preserving diffeomorphisms of $S$ (which may permute the punctures), then the \emph{mapping class group} $\G (S):=\diff^+(S)/\diff^0(S)$ acts on $\Tscr (S)$ by complex-analytic transformations and this action is proper. 
The moduli stack of smooth $n$-punctured curves of genus $g$, $\Mcal_{g,[n]}$, is as an orbifold the  $\G (S)$-orbit space of $\Tscr(S)$.  

A compact 1-dimensional submanifold $A\subset S$ is necessarily a disjoint union of a finitely many embedded circles. 
Say that $A$ is \emph{admissible} if every connected component of $S\ssm A$ is of hyperbolic type (so this includes the case $A=\emptyset$). 
We define the \emph{augmented curve complex} of $S$ as the partially ordered set  $\Ccal^*(S)$ of which an element is an isotopy class $\sigma$  of admissible compact 1-dimensional submanifolds $A\subset S$ as above, the partial order being given by inclusion. Note that $\Ccal^*(S)$ has the empty set as its minimal element (whence `augmented'). For a simplex $\sigma\in \Ccal^*(S)$, we  denote by $\G(S)_{\sigma}\subset \G(S)$ the subgroup that stabilizes this isotopy class  in the strict sense that the isotopy class of each connected component of representative $A$ of $\sigma$  is preserved without reversal of  orientation. This implies that an element of $\G(S)_{\sigma}$ induces a mapping class for each connected component of $S\ssm A$. The Teichm\"uller space $\Tscr (S\ssm A)$ and the  product of the mapping class groups of the connected components of $S\ssm A$ only depend (up to unique  isomorphism) on $\sigma$ and so we take the liberty to write $\Tscr (S\ssm \sigma)$ resp.\ $\G (S\ssm {\sigma})$ instead. The natural homomorphism  $\G(S)_\sigma\to \G (S\ssm \sigma)$ has image a cofinite subgroup of $ \G (S\ssm \sigma)$ and kernel a copy of $\ZZ^{v(\sigma)}$ in $\G(S)_\sigma$, where $v(\sigma)$ is the vertex set of $\sigma$ (a vertex corresponds to the image in  $\G (S)_{\sigma}$ of a Dehn twist along the corresponding component of $A$; beware that $v(\sigma)$ can be empty  in which case $\ZZ^{v(\sigma)}=\{ 0\}$). Note that the image of $\ZZ^{v(\sigma)}$ is a central subgroup of $\G(S)_{\sigma}$. This will be our $\G(S)^\ell_{\sigma}$.

Consider the disjoint union $\overline{\mathscr T} (S)$  of the  Teichm\"uller spaces ${\mathscr T} (S\ssm \sigma)$, where $\sigma$ runs over all the admissible isotopy classes. The group $\Gamma(S)$ acts in this union and there is a natural $\Gamma(S)$-invariant topology on $\overline{\mathscr T} (S)$ which has the property that the closure of ${\mathscr T} (S\ssm \sigma)$ meets  ${\mathscr T} (S\ssm \sigma')$ if and only if $\sigma$ is a face of $\sigma'$. 

The moduli space of stable punctured curves of genus $g$ and with $n$ (unnumbered) punctures, $\Mcalbar _{g,[n]}$ can be regarded as  the $\G (S)$-orbit space of $\overline\Tscr(S)$.  In  fact, $\Mcalbar _{g,[n]}$ is a Deligne-Mumford stack in the complex-analytic category and the stratification of  $\Mcalbar _{g,[n]}$  inherited by that of $\overline\Tscr(S)$  is that of a normal crossing divisor.
It can be shown that every stratum $\Tscr (S\ssm\sigma)$ of $\overline\Tscr (S)$ admits a regular neighborhood $U_\sigma$ in $\overline\Tscr (S)$  whose
$\G(S)$-stabilizer is $\G(S)_\sigma$ and is such that the resulting covering $\{U_\sigma\}_{\sigma\in\Ccal^*(S)}$ of $\overline\Tscr(S)$ satisfies the hypotheses of Theorem \ref{thm:emcov2}.  The theorem in question gives us the following  reformulation of  the theorem of Ebert and Giansiracusa 
\cite{ebertgian} (which for $n=0$ is due to Charney and Lee \cite{charneylee:mg}): 

\begin{theorem}\label{thm:mghomotopy}
The homotopy type of the Deligne-Mumford stack $\Mcalbar_{g,[n]}$ is naturally realized by the classifying space of the category $\Cfrak^*(S)$ whose objects are the elements of the augmented curve complex $\Ccal^*(S)$ and for which a morphism $\sigma\to \sigma'$ is given by a $[\g]\in \ZZ^{\sigma'}\bs \G(S)$ with the property that $[\g]\sigma\subset \sigma'$.
\end{theorem}

We remind the reader that the Deligne-Mumford stack  $\Mcalbar_{g,[n]}$ is not reduced as such when  $(g,n)$ has the value $(0,3)$ (a singleton whose stabilizer is the symmetric group on three elements) or is of  hyperelliptic type $(1,1)$ or $(2,0)$ (then the mapping class group has a center of order two acting trivially).

\section{The homotopy type of a Baily-Borel compactification}
In this section we are going to apply Theorem \ref{thm:emcov2} to a Baily-Borel compactification. To this end we review the basic inputs and properties of that construction, but we have tried to couch these in geometric terms, avoiding the use of root systems.
The point of departure is a connected  linear reductive algebraic group  $\Gcal$ defined over $\QQ$ whose center is anisotropic over $\QQ$ (which means that the Lie  group $G$ underlying $\Gcal (\RR)$ has compact center). We assume that the symmetric space $\XX$ of $G$ (`the space of maximal compact subgroups of $G$') comes with a $G$-invariant complex structure. This turns $\XX$  into a bounded symmetric domain.
We regard $\XX$ as an open subset of its compact dual $\check{\XX}$. This is a complex projective manifold that is homogeneous for $G_\CC$ (the complex Lie group underlying $\Gcal (\CC)$) and the $G_\CC$-stabilizer of a point of $\XX$  is the complexification of its $G$-stabilizer.

\subsection*{Structure of maximal parabolic subgroups}
Let $P$ be a maximal proper parabolic subgroup of $G$ defined over $\QQ$ (i.e.,  the group of real points of such a subgroup of $\Gcal$).  We associate with $P$ the following groups defined over $\QQ$:

\begin{tabular}{r l}
$R_u(P)$\,: &\!\!the  unipotent radical of  $P$.\\
$U_P$\,:  &\!\!the center of $R_u(P)$. This is a  vector group that is never trivial.\\
$V_P$\,: &\!\!the quotient $R_u(P)/U_P$. This is  a (possibly trivial) vector group.\\
$L_P$\,: &\!\!the  \emph{Levi quotient} $P/R_u(P)$ of  $P$. It is a reductive group.\\
$M_P^h$\,: &\!\!the kernel of the action of $L_P$ on $\ufrak_P=\Lie(U_P)$ via the adjoint\\ 
&\!\!representation. The superscript $h$ refers to \emph{hermitian} or \emph{horizontal}.\\
$P^h$\,: &\!\!the preimage of $M_P^h$ in $P$,  in other words,  the kernel of the action\\
&\!\!of $P$ on $\ufrak_P$ via the adjoint representation.\\
$A_P$\,: &\!\!the $\QQ$-split center of $L_P$. This is a copy of $\RR^\times$.\\
$M_P^\ell$\,: &\!\!the commutator subgroup of the centralizer of $M_P^h$ in $L_P$. The \\
&\!\!superscript $\ell$  stands for \emph{link} or \emph{linear}. It has compact center.\\
$L_P^\ell$\,: &\!\!the almost product $M_P^\ell A_P=A_PM_P^\ell$.\\
$P^\ell$\,: &\!\!the preimage of $L_P^\ell$ in $P$.\\
$G(P)$\,: &\!\!the quotient $P/P^\ell=L_P/L_P^\ell$. The composite $M_P^h\subset L_P\twoheadrightarrow G(P)$ is \\
&\!\!onto with finite kernel.\\
\end{tabular}
\newline
Then $P$ acts transitively on $\XX$ and the $P^\ell$-orbits define a holomorphic 
$P$-equivariant fibration of $\XX$, $\pi_P^G:\XX\to \XX (P)$, where $\XX(P)$ is defined as an orbit space. This orbit space  is called a \emph{rational boundary component} of $\XX$ (or rather, of the pair $(\XX, \Gcal)$). It is clear that the $P$-action on $\XX(P)$ is through $G(P)$. This action is transitive  and this realizes  $\XX(P)$ as the bounded symmetric domain associated with $G(P)$. So $\XX(P)$ has its own rational boundary components.

We have  in $\ufrak_P=\Lie (U_P)$  naturally defined a convex open cone $C_P$ that is a $P$-orbit for the adjoint representation. This representation
evidently factors through the Levi quotient $L_P$, but its subgroup  $L_P^\ell=M_P^\ell.A_P$ is still transitive on $C_P$. This cone can be understood as the  $P^h$-orbit space of  $\XX$, the more precise statement being that the semi-subgroup $P^h\exp{(\sqrt{-1}C_P)}\subset G_\CC$ (as acting on $\check{\XX}$)  preserves $\XX$, and makes it in fact an orbit of this semigroup and that we have a $P$-equivariant (real-analytic) bundle $\Iscr_P : \XX \to C_P$ whose fibers are the $P^h$-orbits.  The cone $C_P$ is
self-dual: there is a $P$-equivariant  (but in general nonlinear)  isomorphism of $C_P$ onto its open dual $C_P^\circ\subset\ufrak_P^\vee$,
(i.e.,  the set real forms on $\ufrak_P$ that are positive on $\overline C_P\ssm\{ 0\}$). 

\subsection*{Comparable pairs of parabolic subgroups}
We denote by $\Pscr_{\max}(\Gcal)$  the collection of maximal proper $\QQ$-parabolic subgroups of $\Gcal$ and  identify this set with the corresponding collection of subgroups of $G$.
Since any $P\in \Pscr_{\max}(\Gcal)$ can be recovered from $U_P$ or $\ufrak_P$ as its stabilizer,  a partial order on $\Pscr_{\max}(\Gcal)$ is defined by 
 letting $P\ge Q$ mean that $U_P\supset U_Q$. This is equivalent to: $P^{\ell}\supset Q^{\ell}$ and also to $P^h\subset Q^h$ (but this does not imply that
$R_u(P)\supset R_u(Q)$). 
From the second characterization we see that $P\ge Q$ implies that the projection $\pi_P^G:\XX\to \XX (P)$ factors through $\pi_Q^G:\XX\to \XX (Q)$. The resulting factor $\pi^Q_P: \XX (Q)\to \XX (P)$  then defines a rational boundary component of $\XX (Q)$ of which the associated 
maximal $\QQ$-parabolic subgroup of $G(Q)$ is the image of $P\cap Q$ in  $Q/Q^\ell=G(Q)$. We shall denote that subgroup by $P_{/Q}$.  
The map $P\in \Pscr_{\max}(\Gcal)_{\ge Q} \mapsto P_{/Q}\in \Pscr (\Gcal (Q))$  thus defined is an isomorphism of partially ordered sets.
Note that $P\ge Q$ implies $\XX(P)\le \XX(Q)$.

Let $P, Q\in  \Pscr_{\max}(\Gcal)$ be such that $P\ge Q$. We then have inclusions 
\[
U_Q\subset U_P\cap Q^\ell\subset U_P\subset Q,
\]
where the last inclusion follows from the fact that $U_P$ stabilizes $\ufrak_Q$. The image $U_P/(U_P\cap Q^\ell)$ of $U_P$ in $Q/Q^\ell=G(Q)$ is the
center $U_{P_{/Q}}$ of $R_u(P_{/Q})$ and the projection
\[
c^P_Q: \ufrak_P\to \ufrak_P/(\ufrak_P\cap \qfrak^\ell)\cong \ufrak_{P_{/Q}}.
\] 
maps $C_P$ onto the cone $C_{P_{/Q}}$ that is attached to $P_{/Q}$. This projection fits in a commutative diagram:
\begin{equation}\label{eqn:cone}
\begin{CD}
\XX @>{\pi^G_Q}>> \XX(Q)\\
@V{\Iscr_P}VV  @VV{\Iscr_{P_{/Q}}}V\\
C_P@>{c^P_Q}>>C_{P_{/Q}}
\end{CD}
\end{equation}
Since $\Iscr_P:\XX\to C_Q$ forms the $P^h$-orbit space and $P^h\subset Q^h$, $\Iscr_P$ 
factors through $\Iscr_Q :\XX\to C_P$ and so there is an induced map $\Iscr^P_Q: C_P\to C_Q$. This map is nonlinear in general and is in fact the `adjoint' of the inclusion $C_Q\subset C_P$ via self-duality:  
$C_P\cong C_P^\circ\twoheadrightarrow C_Q^\circ\cong C_Q$. Since $Q^\ell\subset P^\ell$, the adjoint action of  $Q^\ell$ on $\pfrak$ preserves $\ufrak_P$ and $C_P\subset \ufrak_P$. It clearly also preserves the flag of  subspaces $\{0\}\subset  \ufrak_Q\subset\ufrak_P\cap\qfrak^\ell\subset \ufrak_P$
and it will act as the identity on the last quotient $\ufrak_P/(\ufrak_P\cap \qfrak^\ell)\cong \ufrak_{P_{/Q}}$. In fact the map $c^P_Q: C_P\to C_{P_{/Q}}$ is the formation of the $Q^\ell$-orbit space of $C_P$. If we restrict this action of $Q^\ell$  to $R_u(Q)$, then $R_u(Q)$ acts trivially on the successive quotients of this flag and the map 
\[
(\Iscr^P_Q,c^P_Q): C_P\to  C_Q\times C_{P_{/Q}}
\] 
is the formation of the $R_u(Q)$-orbit space of $C_P$. The image of $R_u(Q)$
in $\GL (\ufrak_P)$ is unipotent and this group acts  freely on $C_P$ (this is explained in a more general setting in \S 5 of \cite{looijenga:conepaper}:
in the notation of that paper the above flag is $\{0\}\subset V_F\subset V^F\subset V$, where $V=\ufrak_P$, $C=C_P$ and $F=C_Q$). In particular, 
the map $C_P\to  C_Q\times C_{P_{/Q}}$ is locally trivial with fiber an affine space.

\begin{example}[The symplectic group]\label{example:sp1}
Let $(\Vcal, \langle\, ,\, \rangle)$ be a symplectic vector space over $\QQ$ of dimension $2g$ and take for $\Gcal$ its automorphism group $\Spcal (\Vcal)$. So $G=\Sp(V)$, where $V=\Vcal(\RR)$. The embedding $\sym_2V\hookrightarrow\GLfrak (V)$ which assigns to $a^2\in\sym_2V$ the endomorphism $x\mapsto \langle x, a\rangle a$ maps onto the Lie algebra $\gfrak$ of $\Sp(V)$ and we shall identify the two.

The compact dual $\check{\XX}(V)$ is the space of isotropic complex $g$-planes $F\subset V_\CC$  and  the symmetric domain 
of $\Sp(V)$ is the open subset  $\XX(V)\subset \check{\XX}(V)$ of $F$ on which the Hermitian form $v\in V_\CC\mapsto \sqrt{-1}\langle v, \bar v\rangle\in\CC$ is positive definite.  A maximal proper $\QQ$-parabolic subgroup of $\Sp(V)$ is the $\Sp(V)$-stabilizer (denoted $P_I$) of a nonzero isotropic subspace $I\subset V$ defined over $\QQ$ and vice versa. The associated holomorphic fibration is the  projection $\pi_P : \XX\to \XX(I^\perp/I)$ which sends $F$ to the image of $F\cap I^\perp_\CC\to  (I^\perp/I)_\CC$.

The unipotent radical $R_u(P_I)$ of $P_I$  is the subgroup that acts trivially on $I$ and $I^\perp/I$ (the symplectic form determines an isomorphism  $V/I^\perp\cong I^\vee$ and so this group then automatically acts trivially on $V/I^\perp$).
The center $U_I$ of $R_u(P_I)$ is the subgroup that acts trivially on $I^\perp$ and its (abelian) Lie algebra $\ufrak_I$ can  be identified with 
$\sym_2 I\subset \sym_2V\cong\gfrak$. The cone  $C_I\subset \ufrak_I$ is the cone of positive  definite elements of $\sym_2 I$. The dual cone $C^\circ_I\subset \sym_2 \Hom(I,\RR)$ is the space of positive definite quadratic forms on $I_\RR$ and the duality isomorphism 
$C_I\cong C^\circ_I$ comes from the fact that a positive definite quadratic form on a finite dimensional real vector space determines one on its dual. We identify $R_u(P_I)/U_I$ with a group of elements in  $\GL (I^\perp)$ which act trivially on both $I$ and $I^\perp/I$; this group is abelian  and  its Lie algebra can be identified with $\Hom (I^\perp/I, I)\cong (I^\perp/I)\otimes I$.

The Levi quotient $L_I$ of $P_I$ can be identified $\GL (I)\times \Sp(I^\perp/I)$. The split radical $A_I$ of $L_I$ is the group  of  scalars in $\GL(I)$ (a copy of $\RR^\times$), its horizontal subgroup $M_I^h$  is $\{ \pm 1_I\}\times \Sp (I^\perp/I)$ and its link subgroup $M_I^\ell=\SL(I)$. 
Note that  $G (P_I)=L_I/A_I.M_I^\ell=\Sp (I^\perp/I)$ (which is indeed in an obvious way a quotient of $M_I^h$) and that
$P_I^h$ resp.\ $P_I^\ell$ is the group of symplectic transformations of $V$ that preserve $I$ and act on $I$  as $\pm 1$ resp.\ on $I^\perp/I$  as the identity. 

The projection $\Iscr_I: \XX\to C_I$ is obtained as follows. Let $F\subset V_\CC$ represent an element of $\XX$. Recall that $v\in F\mapsto \half\sqrt{-1}(v, \bar v)$ is a positive definite hermitian form on $F$. The map $F\to (V/I^\perp)_\CC\cong \Hom_\RR(I, \CC)$ is onto with kernel
$F\cap I^\perp_\CC$, so if we identify  $\Hom_\RR(I, \CC)$ with the  orthogonal complement of $F\cap I^\perp_\CC$ in $F$ we get a Hermitian form on $\Hom_\RR (I, \CC)$. The real part of this form defines a positive definite element  of $\sym_2 I$, i.e., an element of $C_I$.

Finally  the partial order relation $P_J\le P_I$ means  simply $J\subset I$. In that case $P_J^\ell$ (the subgroup of $\Sp (V)$ which stabilizes $J$ and acts as the identity on $J^\perp/J$)  indeed preserves $I$  and the image of this action is the full subgroup of $\GL(I)$ which  stabilizes $J$ and acts as the identity on $I/J$. The transformations that also act as the identity on $I$ come from $R_u(P_J)$. The flag defined by $P_J$ in $\ufrak_I=\sym_2I$ is
$\{0\}\subset \sym_2J\subset I\circ J\subset \sym_2I$. If we view $a\in C_I$ as a positive definite quadratic form on $\Hom(I,\RR)$, then the 
subspace $\Hom (I/J, \RR)\subset \Hom(I,\RR)$ has an orthogonal complement with respect to $a$ which maps isomorphically onto $\Hom (J, \RR)$. In other words, 
there is unique section $s$ of $I\to I/J$ and unique
$a'\in C_J$ and $a''\in  C_{I/J}$ such that $s=a'+s_*(a'')$. The resulting projection $C_I\to  C_J\times
C_{I/J}$ is then clearly the formation of the $R_u(P_J)$-orbit space (via which it is a torsor for the vector space $\Hom(I/J, I)$).
The first factor is the nonlinear map $\Iscr^I_J: C_I\to  C_J$ and the 
 second factor is the natural map $c^I_J: C_I\to  C_{I_{/J}}$.
 \end{example}

\subsection*{The  Satake extension} Without loss of generality we may and will assume that $\Gcal$ is almost $\QQ$-simple. We put $\Pscr_{\max}^*(\Gcal):=\Pscr_{\max}(\Gcal)\cup \{\Gcal\}$ and observe that the notions we defined for a member of $\Pscr_{\max}(\Gcal)$ extend in an almost obvious way to $\Pscr_{\max}^*(\Gcal)$. For instance, $R_u(G)=\{1\} $ and so $C_G=\{0\}$, $G(G)=G$ and (hence) $\XX(G):=\XX$.

The \emph{Satake extension} of $\XX$ is a topological space $\XX^\bb$ that contains $\XX$ as an open-dense subset and comes with a stratification:
\[
\XX^\bb=\textstyle{\coprod}_{P\in\Pscr_{\max}^*(\Gcal)} \XX(P),
\]
where the topology on each stratum is the usual one.   For what follows we need a good understanding of the topology on $\XX^\bb$ and so let us briefly review this here.  The incidence relation $\ge $ for the strata will be opposite to the partial order on $\Pscr_{\max}^*$: $\XX(P)\le \XX(Q)$ if and only if $P\ge Q$ (indeed,  the minimal element  $G$ of $\Pscr_{\max}^*$ corresponds to the open subset $\XX=\XX(G)$). So for any $P\in \Pscr_{\max}^*(\Gcal)$, the union of strata containing $\XX(P)$ in its closure is $\Star (\XX (P))= \cup_{Q\le P} \XX (Q)$.
The projections 
$\pi^{Q}_{P} : \XX (Q)\to \XX (P)$ have the property that $\pi_Q^{R}\pi^{Q}_{P} =\pi_{P}^{R}$ when $P\ge Q\ge R$ and hence the $\pi^{Q}_{P}$ combine to form a retraction
\[
\pi_P=:\cup_{Q\le P}\; \pi^{Q}_{P}: \Star (\XX (P)) \to \XX (P)
\]
with the property that $\pi_P\pi_Q=\pi_P|\Star \XX (Q)$ when $Q\le P$.  

The topology on $\XX^\bb$ can be described in terms of cocores.  A \emph{cocore} of $C_P$ (with respect to the $L^\ell_P$-action on $C_P$) is an open subset  $K\subset C_P$ which contains an orbit of an arithmetic subgroup of  $L^\ell_P$ and is such that $\overline{C}_P+K\subset K$. We refer to 
\cite{amrt} for the following basic properties: If $K$ and $K'$ are cocores, then so are $K\cap K'$,  the convex hull of $K\cup K'$ and $\lambda K$ for any $\lambda>0$. Moreover, there exists a 
$0<\lambda_1<\lambda_2$ such that $\lambda_1K\subset K'\subset \lambda_2K$. 
When $Q\le P$, then $c^P_Q$  maps a cocore $K$ in $C_P$ to one in $C_Q$.

For  any cocore $K$, $\Iscr_P^{-1}K$ is invariant under the preimage of $M^h_PM^\ell_P$ in $P$ (this a is normal subgroup of $P$ of codimension one). It maps under $\pi^G_P $ onto $\XX(P)$ and
so 
\[
(\Iscr_P^{-1}K)^\bb:=\textstyle{\coprod}_{Q\le P}\, \pi^G_Q \Iscr_P^{-1}K\subset  \Star (\XX (P))
\]
contains $\XX (P)$. The topology of $\XX^\bb$ at $\XX (P)$ is then characterized by the fact that for every $z\in 
\XX (P)$ the collection $U^\bb(K, V):=(\Iscr_P^{-1}K)^\bb\cap \pi_P^{-1}V$
where $K$ runs over the cocores in $C_P$ and $V$ over the neighborhoods of $z$ in $\XX (P)$ is a neighborhood basis of $z$ in $\XX^\bb$. 
With this topology, $\Star (\XX (P))$ is open subset of $\XX^\bb$, $\XX(P)$ is locally closed in $\XX^\bb$ and the induced topology on $\XX^\bb$ is the one that it already has a symmetric domain. 
It is clear that $\Gcal(\QQ)$ acts on $\XX^\bb$ by homeomorphisms. The space $\XX^\bb$ is Hausdorff, but rarely locally compact. 

\subsection*{Geodesic retraction}The projection $\pi_P$ is a \emph{geodesic retraction}: for every $x\in\XX$ is there is a canonical geodesic  $\g_{P,x}:[0,\infty)\to \XX$ through that point with $\lim_{t\to\infty}  \g_{P,x}(t)=\pi_P(x)$. A geodesic through $x$ is given by
a one-parameter subgroup of $G$ that is `perpendicular' to the compact subgroup $G_x$; in the present case it is the one in $P$ whose projection in $L_P$ is given by the action of $A_P$. The  image of this geodesic under the projection $\Iscr_P: \XX\to C_P$ is then just the ray that lies on the line spanned by $\Iscr_P(x)$.
These geodesics are defined on all of  $\Star (\XX (P))$ (albeit that they will be constant on $\XX(P)$) and  depend continuously on their point of departure. So this defines a $(\G\cap P)$-equivariant deformation retraction of $\Star (\XX (P))$ onto $\XX(P)$. (We can now also be explicit about the map $\Iscr: \XX\to  C^\circ_P$: fix a $G$-invariant hermitian metric on $\XX$. For every $u\in\ufrak_P$, denote by
$u_x\in  T_x\XX$ the infinitesimal displacement defined by the action. Then  $\Iscr (x)(u)$ is the imaginary part of 
$h_x(u_x,\dot\g_{P,x}(0))$.)

Since a cocore $K$ of $C_P$ is invariant under  multiplication by scalars $\ge 1$, the geodesic deformation retraction preserves $(\Iscr_P^{-1}K)^\bb$ and so restricts to one of $(\Iscr_P^{-1}K)^\bb$ onto $\XX(P)$.
For any $Q\le P$, we have $K+C_Q\subset K$ and from this one may deduce that the deformation retraction of $\Star (\XX (Q))$ onto 
$\XX(Q)$ also preserves $(\Iscr_P^{-1}K)^\bb\cap \Star (\XX(Q))$. Moreover, the diagram (\ref{eqn:cone}) specializes to 
\[
\begin{CD}
\Iscr_P^{-1}K @>{\pi^G_Q}>> \Iscr_{P_{/Q}}^{-1}K_{P_{/Q}}\\
@V{\Iscr_P}VV  @VV{\Iscr_{P_{/Q}}}V\\
K@>{c^P_Q}>>K_{P_{/Q}}
\end{CD}
\]
where $K_{P_{/Q}}:=c^P_Q(K)$ is a cocore in $C_{P_{/Q}}$.
Since the top arrow is onto, it follows that $\Iscr_{P_{/Q}}^{-1}K_{P_{/Q}}=(\Iscr_P^{-1}K)^\bb\cap \XX(Q)$. In particular, every stratum of $(\Iscr_P^{-1}K)^\bb$ is given by a cocore.

\subsection*{The Baily-Borel compactification}
Suppose $\G\subset \Gcal (\QQ)$  is an arithmetic subgroup. The central result of the  Satake-Baily-Borel theory asserts that the orbit space 
$\G\bs \XX^\bb$  is a compact topological space, the \emph{Baily-Borel compactification} of $\G\bs \XX$, which underlies the structure of a 
complex projective variety. Note that $Z_\G (\XX(P))$ contains  $\G\cap P^\ell$ as a subgroup of finite index.
A key step in the proof is the local version which states that the orbit space $(\G \cap P^\ell)\bs \Star (\XX (P))$ is locally compact (and has in fact the 
structure of normal complex-analytic variety). The $(\G\cap P)$-equivariant  geodesic deformation retraction $\pi_P$ descends to a 
$\G (P)$-equivariant geodesic deformation retraction $(\G\cap P^\ell)\bs \Star (\XX (P))\to \XX (P)$. 

The image  of $\G\cap P$ in $L^\ell_P$ is an arithmetic subgroup and so there exist cocores $K_P$ in $C_P$ that are invariant under the image of 
$\G\cap P$ in $L^\ell_P$. For such a cocore,  $U_{K_P}:=\Iscr_P^{-1}K_P$ is of course invariant under $\G\cap P$ and what we just asserted about 
$\Star (\XX (P))$ also holds for $U_{K_P}^\bb$. In particular, $U_{\XX(P)}(K):=(\G\cap P^\ell)\bs U_{K_P}^\bb$  can be regarded as a regular 
open neighborhood of $\XX(P)$  in $(\G\cap P^\ell)\bs\XX^\bb$. The retraction $\pi_P$  induces a $\G (P)$-equivariant geodesic deformation retraction  
$U_{\XX(P)}(K)\to \XX(P)$. Since $\XX(P)$ is contractible, so will be $U_{\XX(P)}(K)$. 

We can take $K_P$ so small as  to ensure that every $(\G\cap P)$-orbit in $U_{K_P}^\bb$ is the intersection of $U_{K_P}^\bb$ with a $\G$-orbit. 
This implies that if for some $\g\in\G$, $\g U_{K_P}^\bb$ meets $U_{K_P}^\bb$, then $\g\in P$ and in particular $\g U_{K_P}^\bb=U_{K_P}^\bb$.  
So for a stratum $S=\G(P)\bs\XX(P)$ of $\G\bs \XX^\bb$,  $U_S(K):=(\G\cap P)\bs U_{K_P}^\bb=\G(P)\bs U_{\XX(P)}(K)$ is a regular open neighborhood of $S$  in 
$\G\bs\XX^\bb$ and $\pi_P$ will induce a deformation retraction of  $U_S(K)$ onto $S$.

\begin{definition}
Let $\G\subset \Gcal (\QQ)$ be a subgroup.
The \emph{Satake category} $\Sfrak_{\G}$ of the pair $(\Gcal, \G)$ is the small category whose object set is $\Pscr_{\max}^*(\Gcal)$ and for which 
a morphism $P\to P'$ is given by a $\gamma\in \G$ with the property that ${}^\gamma P (:=\gamma P \gamma^{-1})\le P'$ (or equivalently,  
$\g\XX (P)\ge \XX (P')$). The \emph{Charney-Lee category} $\Wfrak_\G$ of $\Sfrak_{\G}$ has the same object set, but a $\Wfrak_{\G}$-morphism 
 $P\to P'$ is given by a right coset $(\G\cap P'{}^\ell)\gamma\in (\G\cap P'{}^\ell)\bs \G$ with the  property that  ${}^\gamma P \le P'$.
\end{definition}

So a  $\Wfrak_\G$-morphism  $P\to P'$ is almost tantamount to giving  a rational boundary component $\XX(Q)\le \XX(P)$ plus  an isomorphism of  $\XX(Q)$ onto $\XX(P')$ that is induced by an element of $\G$.

We have an obvious functor $F: \Sfrak_{\G}\to \Wfrak_\G$. The fiber of the identity of $P\in \Pscr_{\max}^*(\Gcal)$ in $\Sfrak_{\G}$,  
when  viewed as an object of $\Wfrak_\G$ is equal to $\G\cap P^\ell$. It is clear that for any subgroup  $\G_1\subset \G$,  
$\Sfrak_{\G_1}$ resp.\  $\Wfrak_{\G_1}$ appears as a subcategory of $\Sfrak_{\G}$  resp.\ $\Wfrak_{\G}$.

\begin{theorem}\label{thm:main}
Let $\G$ be an arithmetic subgroup of $\Gcal (\QQ)$. The classifying space functor applied to the embedding of $\G$ in $\Sfrak_{\G}$ (as the automorphism group of the object defined by $\Gcal$) is a homotopy equivalence and so $|B\Sfrak_{\G}|$ represents the homotopy type of the
Deligne-Mumford stack  $\G\bs \XX$. The Baily-Borel compactification $\G\bs \XX^\bb$ of $\G\bs \XX$ comes with a natural structure of a stacky homotopy type that is represented by $|B\Wfrak_\G|$
such that  the classifying space construction applied to the  functor $\Sfrak_{\G}\to \Wfrak_\G$
reproduces the stacky homotopy type of the inclusion $\G\bs \XX\subset \G\bs \XX^\bb$.  In particular, the rational cohomology algebra of $\G\bs \XX^\bb$ is that of $|B\Wfrak_\G|$.
\end{theorem}

\begin{example}[Example \ref{example:sp1} continued]\label{example:sp2}
An object of $\Sfrak_{\G}$ is then given by an isotropic subspace $I\subset V$ and a morphism $I\to J$ by a $\gamma\in\G$ such that
$\gamma I\subset J$. Two such elements $\gamma, \gamma'\in\G$ define the same morphism in the Charney-Lee category precisely if $\gamma' \gamma^{-1}$ preserves $J$ and induces the identity in $J^\perp/J$. 
 \end{example}
 
\begin{proof}[Proof of theorem \ref{thm:main}]
We regard $\G\bs\XX$ as a quotient stack so that its homotopy type as such is given by $B\G$.
Next we observe that the forgetful functor  $R: \Sfrak_{\G}\to \G$ (which forgets $P$) is a retract. 
The fiber of $R$ over the identity of $G$ is the subcategory of $\Sfrak_\pt$ defined by the finite  linear chains in  $\Sfrak_\pt$ that have $G$ as a minimal element. This category has an initial object, namely the identity of $S$ (now viewed as a linear chain of length zero).
This implies that $|B (\G\bs R)|$ is contractible so that  by Thm.\ A of \cite{quillen}, $|BR|$ is a homotopy equivalence.  

The remaining assertions will follow if we verify the hypotheses of Theorem \ref{thm:emcov2} for $\XX^\bb$ with its natural stratification into $\XX(P)$ and take for $\G^\ell_{\XX(P)}:=\G\cap P^\ell$ as our link group. Then $\G_{\XX(P)}/\G^\ell_{\XX(P)}=\G\cap P/\G\cap P^\ell=\G (P)$ acts properly on $\XX(P)$ with a subgroup of finite index acting freely. Since $\XX(P)\le\XX(Q)$ is equivalent to $P^\ell\supset Q^\ell$, we  then have $\G^\ell_{\XX(P)}\supset \G^\ell_{\XX(Q)}$, as required.

For every $P\in \Pscr_{\max}^*(\Gcal)$ we choose in a $\G$-equivariant fashion  an open cocore  $K_P\subset C_P$ (meaning that 
$K_{\gamma P\gamma ^{-1}}=\gamma K_P$). We let $U_{K_P}:=\Iscr_P^{-1}K_P$ and $U_{K_P}^\bb$ be as before. We know that
$U_{K_P}^\bb$ is then open in $\XX^\bb$ and contains $\XX (P)$ as a $(\G\cap P)$-equivariant deformation retract. It is then clear that 
these neighborhoods satisfy properties (i) and  (ii) of Theorem \ref{thm:emcov2}. 

We noted that the orbit space $U_{\XX(P)}(K)=(\G\cap P^\ell)\bs U_{K_P}^\bb$ is an analytic variety with $\G(P)$-action which comes with an analytic $\G(P)$-equivariant retraction $U_{\XX(P)}(K)\to  \XX(P)$.
The group $\G(P)$ acts on $\XX(P)$ as an arithmetic group and hence  this  action is proper with a subgroup of finite index acting freely. The same is then true
for its action on $U_{\XX(P)}(K)$  and so property (iii) is also satisfied.

 On order to check property (iv), consider any chain $P_\pt=(P_0<P_1<\cdots <P_n)$ in $\Pscr_{\max}(\Gcal)$ and
put $U^\bb_{K_{P_\pt}}=\cap_{i=0}^n U^\bb_{K_{P_i}}$. We must show that
$(\G\cap P^\ell)\bs U^\bb_{K_{P_\pt}}$ is contractible. 
For any $x\in U^\bb_{K_{P_\pt}}$, the geodesic $\g_{P_0,x}$ stays in $U^\bb_{K_{P_\pt}}$ and so we have a $(\G\cap P_0^\ell)$-equivariant deformation retraction of $U^\bb_{K_{P_\pt}}$  onto its intersection with $\XX(P_0)$. In particular,  $(\G\cap P^\ell)\bs U^\bb_{K_{P_\pt}}$   has $U^\bb_{K_{P_\pt}}\cap \XX(P_0)$ as deformation retract. 

Since we are now left to prove that  $U^\bb_{K_{P_\pt}}\cap \XX(P_0)$ is contractible, we focus on $\XX(P_0)$ with its $\G(P_0)$-action. This means
that we can pretend that $P_0=G$, so that we must show that $\cap_{i=1}^nU_{K_{P_i}}$ is a contractible subset of $\XX$. The chain $P_\pt$ defines a `flag' of faces $\{0\}=C^+_{P_0}\subset C^+_{P_1}\subset\cdots \subset C^+_{P_n}$.  But then $\cap_{i=1}^n U_{K_{P_i}}$ is equal to 
$U_K$, where $K:=\cap_{i=1}^n (\Iscr^{P_n}_{P_i})^{-1} K_{P_i}\subset C_{P_n}$. 
So it remains to prove that $K$ is contractible: for then so is $U_K$ and we then apply Theorem \ref{thm:emcov2}. 

To this end we write $P$ for $P_n$ and $Q$ for $P_1$. Since $K_{P_i}$ is invariant under $P_i^\ell$, it is also invariant  under $R_u(Q)$ (for $Q\le P_i$). 
Hence $K$ is $R_u(Q)$-invariant. Since  $(\Iscr^P_{Q},c^P_{Q}): C_P\to C_{Q}\times C_{P_{/Q}}$ forms the $R_u(Q)$-orbit space and has 
affine fibers, it suffices to prove that the image of $K$ under this map is contractible.  This image is open and invariant under  translations in 
the convex cone $C_Q\times \{0\}$ and projects in $C_{P_{/Q}}$ onto an open subset invariant under  translations in $C_{P_{/Q}}$. A double application of Lemma \ref{lemma:contractiblecocore} below then finishes the proof. \end{proof}

\begin{lemma}\label{lemma:contractiblecocore}
Let $U$ and $U'$ be a real finite dimensional vector spaces, $C\subset U$ an open  convex cone and $K\subset C\times U'$ an open subset
which is invariant under translations in $C\times\{ 0\}$. Then the projection  $K\xrightarrow{\pi_{U'}} \pi_{U'}(K)$ is a homotopy equivalence.
\end{lemma} 
\begin{proof}
With loss of generality we may assume that $C$ is nondegenerate. Put $K':= \pi_{U'}(K)$ and 
choose $\phi\in C^\circ$. Then the base $\PP(C)$ is a convex open subset of the affine subspace of $\PP(U)$ defined by $\phi\not= 0$ and 
so  $\PP(C)$ is contractible. For every $r\in \PP(C)$ and $y\in K'$ denote by $\lambda (r,y)>0$ the infimum of $\phi$ on the intersection of the ray emanating from $(0,y)$ defined by $r$ with $K$. Then $\lambda$ is continuous and if $p: C\to \PP(C)$ is the obvious projection, then 
$(p, \phi)$ maps $K$ homeomorphically onto the subspace of $\PP(C)\times K'\times (0, \infty)$ consisting of $(r,y,t)$ with $t>\lambda(r,y)$. 
The projection of this image onto $\PP(C)\times K'$ is a clearly a homotopy equivalence. And so is the projection of $\PP(C)\times K'$  onto $K'$.
\end{proof}

\begin{remark}\label{rem:titsbuilding}
We  recall that $\Pscr_{\max}(\Gcal)$ is the vertex set of the \emph{Tits building}  of the $\QQ$-group $\Gcal$.  This is a simplicial complex whose simplices are the linear chains in $\Pscr_{\max}(\Gcal)$ (and so any simplex comes with a total order on its vertex set).  To give such a  linear chain $\Pcal_\pt=(P_0< P_1< \cdots <P_k)$ amounts to giving a proper $\QQ$-parabolic subgroup of $G$ (namely $\cap_i P_i$ ), for  if $P$ is a proper $\QQ$-parabolic subgroup of $G$, then the collection of maximal proper $\QQ$-parabolic subgroups containing $P$ is  a chain in $\Pscr_{\max}(\Gcal)$ and the intersection of its members give us back $P$. In other words, the collection of nonempty linear chains in $\Pscr_{\max}(\Gcal)$ can be identified with
the collection of proper $\QQ$-parabolic subgroups of $G$, even as partially ordered sets,  where the relation \emph{`is a subchain of'} corresponds to
the relation \emph{`contains'}.
\end{remark}

\section{The Satake compactification of $\Acal_g$ according to Charney-Lee}
Denote by $\Vfrak_g$  the category whose objects are pairs $(L\supset I)$, where $L$ is a unimodular symplectic lattice of rank $2g$ and  $I\subset L$ is primitive isotropic sublattice and for which a morphism $(L\supset I)\to (L'\supset I')$ is given by an isomorphism $\phi :L\cong L'$ such that
$\phi (I)\subset I'$. Letting $\mathfrak{Sp}_g(\ZZ)$ denote the groupoid of unimodular symplectic lattices $L$ of rank $2g$ whose morphisms are  symplectic isomorphisms, then we have a forgetful functor $\Vfrak_g\to \mathfrak{Sp}_g(\ZZ)$ defined by $(L\supset I)\mapsto L$. This   is also a homotopy equivalence, because a fiber over $L$ is the PO-set of primitive isotropic sublattices and this has an initial object (namely $0$), so has a contractible geometric realization. Let us write $H$ for the lattice $\ZZ^2$ equipped with its standard symplectic form. Since every unimodular symplectic lattice of rank $2g$ is isomorphic to $H^g$, the full subcategory $\Sp (H^g)\subset \mathfrak{Sp}_g(\ZZ)$ is an equivalence and so the inclusion
$Sp (H^g)\subset \Vfrak_g$ (definined by taking $I=0$ in $H^g$) yields a homotopy equivalence after passing to classifying spaces.

The \emph{Giffen category of genus $g$,} $\Wfrak_g$, is the  category whose objects are  the unimodular symplectic lattices $M$ of rank $\le 2g$ and  for which  a morphism $M\to M'$ is  given by a primitive isotropic sublattice $I\subset M$ and a symplectic  isomorphism  $I^\perp/I\xrightarrow{\cong}M'$ (the composition should be clear). A functor  $F_g: \Vfrak_g\to \Wfrak_g$ is defined by $F_g(L\supset I):= I^\perp/I$. Indeed, for a $\Vfrak_g$-morphism 
$\phi :(L\supset I)\to (L'\supset I')$, we have $I\subset \phi^{-1}I'$ and $J:=\phi^{-1}I'/I$ is then an isotropic subspace of $F_g(L\supset I)=I^\perp/I$ such that $\phi$ induces an isomorphism of $J^\perp/J$  onto $I'{}^\perp/I'=F_g(L'\supset I')$.

We now consider a special case of Example \ref{example:sp1}. We take as our $\QQ$-algebraic group the group $\Spcal_g$ which assigns to a commutative ring  $R$ with unit the group $\Sp (R\otimes H^g)$ so that
$\Sp(H^g)$ is an arithmetic subgroup of $\Spcal_g(\QQ)=\Sp (H^g_\QQ)$. The associated real Lie group  $\Spcal_g(\RR)=\Sp (H^g_\RR)$ has as its symmetric space the domain $\XX_g:=\XX(H^g)$ and  $\Sp(H^g)\bs \XX_g $ can be identified with the moduli space $\Acal_g$ of principally polarized abelian varieties. It is clear that  $\Sfrak_{\Sp(H^g)}$ is the full subcategory of $\Vfrak_g$ whose objects are of the form $(H^g\supset I)$. 
The interpretation of $\Wfrak_{\Sp (H^g)}$ as in  Example \ref{example:sp2}  enables us to define a  functor $\Wfrak_{\Sp (H^g)}\to \Wfrak_g$ by $I\mapsto I^\perp/I$. We then have a commutative diagram of functors
\[
\begin{CD}
\Sfrak_{\Sp(H^g)}@>>>\Vfrak_g\\
@VVV  @VV{F_g}V\\
\Wfrak_{\Sp (H^g)}@>>> \Wfrak_g
\end{CD}
\]
where the vertical arrow on the left is given by Theorem \ref{thm:main}. Since every unimodular symplectic lattice of rank $2g$ is isomorphic to $H^g$, the horizontal arrows are equivalences of categories and so Theorem \ref{thm:main} gives the following rephrasing of  a theorem of  
Charney-Lee \cite{charneylee:satake}:

\begin{proposition}
The  inclusion $\Sp (H^g)\subset \Vfrak_g$ is an equivalence of categories and the stacky homotopy type of the inclusion of $j_g: \Acal_g\subset \Acal_g^\bb$ is reproduced  by applying the classifying space construction applied to the  functor $F_g: \Vfrak_g\to \Wfrak_g$. 
\end{proposition}

\begin{remark}\label{rem:stabmaps}
There is a monoidal structure present that we wish to explicate in view of its applications to cohomological stability \cite{chen-looijenga}. The map which assigns to two principally polarized abelian varieties  their product defines a morphism $\Acal_g\times \Acal_{g'}\to \Acal_{g+g'}$. This morphism is covered by the  map
$\XX_g\times \XX_{g'}\to \XX_{g+g'}$  which assigns to the pair $(F\subset H^g_\CC, F'\subset H^{g'}_\CC)$ the direct sum $F\oplus F'\subset H^{g+g'}_\CC$. The corresponding functor  $\Vfrak_g\times \Vfrak_{g'}\to \Vfrak_{g+g'}$ is given by  $((L\supset I), (L'\supset I'))\mapsto (L\oplus L', I\oplus I')$.
The map $\XX_g\times \XX_{g'}\to \XX_{g+g'}$ extends in an obvious manner to the Satake extensions $\XX^\bb_g\times \XX^\bb_{g'}\to \XX^\bb_{g+g'}$ and hence drops to a morphism $\Acal^\bb_g\times \Acal^\bb_{g'}\to \Acal^\bb_{g+g'}$ that extends the map $\Acal_g\times \Acal_{g'}\to \Acal_{g+g'}$ above. Its counterpart  $\Wfrak_g\times \Wfrak_{g'}\to \Wfrak_{g+g'}$ for the Giffen categories is given $(M, M')\mapsto M\oplus M'$. Indeed, the  commutative diagram on the right below has the same rational homology type as the commutative diagram on the left.
\begin{equation}\label{eqn:monoid}
\begin{CD}
\Acal_g\times \Acal_{g'} @>>> \Acal_{g+g'}\\
@V{j_g\times j_{g'}}VV  @VV{j_{g+g'}}V\\
\Acal^\bb_g\times \Acal^\bb_{g'} @>>> \Acal^\bb_{g+g'}
\end{CD}
\hskip 30mm
\begin{CD}
\Vfrak_g\times \Vfrak_{g'} @>>> \Vfrak_{g+g'}\\
@V{F_g\times F_{g'}}VV  @VV{F_{g+g'}}V\\
\Wfrak_g\times \Wfrak_{g'} @>>> \Wfrak_{g+g'}\\
\end{CD}
\end{equation}
By taking $g'=1$ and choosing a point of $\Acal_1$ resp.\  the element $(H,I)$, where $I$ is the span of the first basis vector of $H$, the above diagrams become the `stabilization maps'
\begin{equation}\label{eqn:stable}
\begin{CD}
\Acal_g@>>> \Acal_{g+1}\\
@V{j_g}VV  @VV{j_{g+1}}V\\
\Acal^\bb_g @>>> \Acal^\bb_{g+1}
\end{CD}
\hskip 30mm
\begin{CD}
\Vfrak_g@>>> \Vfrak_{g+1}\\
@V{F_g}VV  @VV{F_{g+1}}V\\
\Wfrak_g @>>> \Wfrak_{g+1}\\
\end{CD}
\end{equation}
The homotopy type of the maps on the right hand side do not depend on the point we choose, for $\Acal_1$ is isomorphic to the affine line and hence connected. 
\end{remark}

\subsection*{The homotopy type of the extended period map} The  map which assigns to a compact Riemann surface of genus $g>1$  its Jacobian as a principally polarized abelian variety defines a \emph{period map} $\Jscr:\Mcal_g\to \Acal_g$. If $S_g$ is a  closed connected oriented surface, then the $\QQ$-homotopy type of $\Jscr$  is represented by the map on classifying spaces of the group homomorphism $\G (S)\to \Sp (H_1(S))$.  Mumford observed that the  period map $\Jscr:\Mcal_g\to \Acal_g$ extends to a morphism $\Jscr^\bb:\Mcalbar_g\to \Acal_g^\bb$. It has the property that the preimage of a  stratum of $\Acal_g^\bb$ is a closed union of strata of $\Mcalbar_g$.

\begin{proposition}\label{prop:perfecttorelli}
Let $S_g$ be a  closed connected oriented surface of genus $g>1$.  Let  $P: \Cfrak^*(S_g) \to \Wfrak_g$ be the functor which assigns to an element $\sigma$ of the augmented curve complex $\Ccal^*(S_g)$, the quotient  $\overline H_1(S_g\ssm \sigma)$ of the quasi-symplectic lattice $H_1(S_g\ssm \sigma)$ by its radical (or equivalently, the image of $H_1(S_g\ssm \sigma)\to H^1(S_g\ssm \sigma)$). The restriction of this functor to the initial object $\emptyset$ of $\Cfrak^*(S_g)$ gives the symplectic representation  $P_\emptyset: \G (S)\to\Sp (H_1(S))$  and   the stacky homotopy type of the square on the left below  is obtained by applying the classifying space functor to the square on the right:
\[
\begin{CD}
\Mcal_g @>{\Jscr}>> \Acal_g\\
@VVV  @VVV\\
\Mcalbar_g @>{\Jscr^\bb}>> \Acal^\bb_g
\end{CD}
\hskip 30mm
\begin{CD}
\G (S) @>{P_\emptyset}>> \Sp (H_1(S))\\
@VVV  @VVV\\
\Cfrak^*(S_g) @>{P}>> \Wfrak_{\Sp (H_1(S))}
\end{CD}
\]
\end{proposition}
\begin{proof}We confine ourselves to the basic idea of the proof. First note that the  period map lifts to a map $\Tscr (S_g)\to \XX(H_1(S_g))$. This extends to a continuous map $\overline{\Tscr} (S_g)\to \XX(H_1(S_g))^\bb$ which on the stratum $\Tscr (S_g\ssm\sigma)$ it is given by first mapping $\Tscr (S\ssm\sigma)$ to the Teichm\"muller space of the (possibly disconnected) surface obtained from $S\ssm\sigma$ by filling in all the punctures and then applying the period map on each connected component. We can arrange that the open cover of  $\overline{\Tscr} (S_g)$ that was used to define $\Cfrak^*(S_g)$
refines the preimage of the open cover of $\XX(H_1(S_g))^\bb$ that was used to define its Charney-Lee category.
The proposition then follows.
\end{proof}

\section{The homotopy type of a toroidal compactification}\label{section:toroidal}

\subsection*{The parabolic cone}
We place ourselves in the setting of the previous section. Let us first recall from \cite{amrt} how a toroidal compactification is defined. 
Let  $\gfrak$ stand for the $\QQ$-Lie algebra of $\Gcal$ and regard
$C_P$ as a  cone in $\gfrak(\RR)$. Then any element of  $\gfrak(\QQ)$ in the closure of $C_P$ lies in a \emph{unique} $C_Q$ with $Q\le P$ and if we
define the \emph{parabolic cone} as $C(\gfrak):=\cup_{P\in \Pcal^*_{\max}(\Gcal)} C_P$ and define the 
\emph{face (of $C(\gfrak)$) associated to $P$} as $C_P^+:=\cup_{Q\le P} C_Q$, then
\begin{enumerate}
\item [(a)] $C(\gfrak):=\sqcup_{P\in \Pcal^*_{\max}(\Gcal)} C_P$ (the union is disjoint),
\item [(b)] $C_P^+$ is the relative closure of $C_P$ in $C(\gfrak)$ and $P\le Q$ if and only if $C_P^+\le C_Q^+$, 
\item [(c)]  two faces  intersect in a face. 
\end{enumerate}
So the faces of $C(\gfrak)$ are in bijective correspondence with the elements of $\Pcal^*_{\max}(\Gcal)$ and the flags of faces that are not reduced to $\{0\}$ are in bijective correspondence with simplices of the Tits building of $\Gcal$.

For every $P\in \Pcal^*_{\max}(\Gcal)$, the group $\G\cap P$ acts via an arithmetic subgroup of 
$L^\ell_{P}$ on $\ufrak_P$  and preserves $C_P^+$. It is known to have as fundamental domain in $C_P^+$ a rational polyhedral cone (i.e., the convex cone spanned by a finite subset of $\ufrak_P(\QQ)$). For example, if $\phi\in 
\ufrak_P^*$ is such that $\phi$ is positive on $\overline C_P-\{0\}$, then the set of $x\in C_P^+$ with 
$\phi (x)\le \phi (\gamma x)$ for all $\gamma\in \G\cap P$ is a rational polyhedral cone that is also a fundamental domain for  $\G\cap P$.
So if $\Sigma_P$ is a $\G\cap P$-invariant decomposition of $C_P^+$ into  rational polyhedral cones, then it induces one on each of its faces $C_Q^+$, $Q\le P$. 

\subsection*{Admissible decompositions of the parabolic cone}
Let $\Sigma$ be a $\G$-invariant decomposition of $C(\gfrak)$ into a rational polyhedral cones (such decomposition is said to be \emph{$\G$-admissible}).
This determines a  toroidal  extension of $\XX^\Sigma$ of $\XX$ which is locally like the one we have for  the  extension described in  the torus case \ref{example:torus} and is at the same time very much in the spirit of  the Satake-Baily-Borel extension. The difference with the latter is that the projections  $\pi_P$ are replaced by projections  $\XX\to \XX(\sigma)$ indexed by the cones $\sigma\in\Sigma$  for which the topology is easier to understand.  A fiber of this projection  is an orbit of the semigroup $\exp (\langle  \sigma\rangle_\RR +\sqrt{-1}(\langle  \sigma\rangle_\RR\cap C_P)$ acting on $\check{\XX}$. Let
$\pi^{\{0\}}_\sigma : \XX\to \XX(\sigma)$ denote the formation of this orbit space. Then $\XX(\sigma)$ has the structure of  a complex manifold for which  $\pi^{\{0\}}_\sigma$ is a holomorphic map.  The $\G$-stabilizer $\G_\sigma$ of $\sigma$ acts on $\XX(\sigma)$ with a  kernel that contains the free abelian group $\G\cap \exp (\langle  \sigma\rangle_\RR)$ as a subgroup of finite index. We shall take $\G^\ell_{\XX(\sigma)}:=\G\cap \exp (\langle  \sigma\rangle_\RR)$.

For any $\sigma\in \Sigma$, we denote by $P(\sigma)$ the  member of $\Pscr_{\max}^*(\Gcal)$  with the property that $\sigma\cap C_{P(\sigma)}\not=\emptyset$. Then $\pi^G_{P(\sigma)}$ factors through $\pi^{\{0\}}_\sigma$. This is in fact a principal fibration over $\XX(P(\sigma))$ with structure group is an extension of the vector group $V_{P(\sigma)}$  by the vector group $U_{{P(\sigma)},\CC}/\exp (\langle  \sigma\rangle_\CC)$). In particular, $\XX(\sigma)$ is contractible. Notice that $\pi^{\{0\}}_{\{0\}}:\XX\to \XX(\{ 0\})$ is the identity map.

 For $\sigma\ge  \tau$ we have a factorization of $\pi^{\{0\}}_\sigma$ over $\pi^{\{0\}}_\tau$. The factor $\pi^\tau_\sigma: \XX(\sigma)\to 
\XX(\tau)$ is a holomorphic fibration  and  we have $\pi^\upsilon_\tau\pi^\tau_\sigma=\pi^\upsilon_\sigma$ when  $\sigma\le\tau\le \upsilon$.
We then proceed as before  by putting  $\Star_\Sigma \XX(\sigma):= \coprod_{\tau\le \sigma} \XX(\tau)$ so that the $\pi^\tau_\sigma$ combine to give 
a retraction $\pi_\sigma:\Star_\Sigma \XX(\sigma)\to\XX(\sigma)$.  We let $\XX^\Sigma$  be the disjoint union of the $\XX(\sigma)$ and equip this 
union with the topology that is similarly defined as in the Satake-Baily-Borel  setting,  the role of the cocores   then being taken by an open 
convex subsets $K\subset C_P$ that can be written as the Minkowski sum of $\sigma$ and a nonempty bounded open subset of $C$.  
The  $\G$-orbit space $\G\bs \XX^\Sigma$ is a compact Hausdorff space  (which in fact underlies the structure of a normal analytic space) 
and the obvious  $\G$-equivariant  map $\XX^\Sigma\to \XX^\bb$ is continuous and yields a  surjective map $\G\bs\XX^\Sigma\to \G\bs\XX^\bb$ 
of compact Hausdorff spaces (which in fact underlies a morphism in the complex-analytic category).

The  group $\G (\XX(\sigma)):=\G_\sigma/(\G\cap \exp (\langle  \sigma\rangle_\RR))$ acts properly on $\XX(\sigma)$. 

\begin{theorem}\label{thm:main2}
Let $\Tfrak^\Sigma_\G$ be the category with objects the members of $\Sigma$ and for which a morphism $\tau \to \sigma$ is  given by a right coset 
$\G\cap \exp (\langle  \sigma\rangle_\RR)\gamma\in (\G\cap \exp (\langle  \sigma\rangle_\RR)\bs \Gamma$ for which 
$(\G\cap \exp (\langle  \sigma\rangle_\RR))\gamma\tau \subset\sigma$. Then the full subcategory of $\Tfrak^\Sigma_\G$ defined by the  object  
$\{ 0\}$ can be identified with $\G$ and we have a natural functor  $\Tfrak^\Sigma_\G\to\Wfrak_\G$  defined by $\Pi\mapsto P(\Pi)$. If we apply the 
classifying space construction to the functors  $\G \subset \Tfrak^\Sigma_\G\to \Wfrak_\G$ we recover the stacky homotopy type of the morphisms 
$\G\bs \XX\subset \G\bs \XX^\Sigma\to \G\bs \XX^\bb$. \end{theorem}

\subsection*{An example: the perfect cone compactification.} 
We take $\G=\Sp (H^g)$. The perfect cone  compactification of $\Acal_g^\perf$ of $\Acal_g$ is an example of a toroidal compactification as above. 

\begin{definition}
Given a lattice $L$, denote by $\sym_2(L)$ the symmetric quotient of $L\otimes L$ and by $\Sq(L)$ the collection of \emph{pure primitive squares} in $\sym_2(L)$, i.e.,  elements of the form $v^2$ with 
$v\in L$  primitive. We say that a finite subset $\Pi\subset \Sq(L)$ is \emph{perfect} if it is the intersection of $\sym_2(L)$ with  a face of the convex hull of $\Sq(L)$ in $\sym_2 L_\RR$, agreeing that the empty set is also perfect.
\end{definition}

A duality property for locally polyhedral convex sets (the convex hull of $\Sq(L)$ is one) implies that this is also equivalent to the existence of a linear form 
$\sym_2(L)\to \RR$ with the property that it is $\ge 1$ on each pure primitive square, with the value $1$ taken if and only if the pure primitive square
is in $\Pi$. By regarding  such a linear form as a  quadratic form on $L_\RR$, we see that a subset $\Pi\subset \Sq(L)$ is perfect if and only if there exists a
positive definite quadratic form $q$ on $L$  such that $v^2\in\Pi$ if and only if  $q|L-\{ 0\}$ takes its minimal value in $v$.  With this characterization it is easy to show that (i) $\Pi$ is the vertex set of its convex hull in  $\sym_2 L_\RR$: no element of $\Pi$ is a convex linear combination of the others, and 
(ii) the property of $\Pi$ being perfect is in a sense independent of $L$: it only depends on smallest sublattice  $I\subset L$ with $\Pi\subset \Sq (I)$ (we denote that sublattice $I_\Pi$). If we denote by $J_\Pi\subset L_\RR$ the biggest subspace of $L_\RR$  such that 
$\sym_2(J_\Pi)\subset \la\Sq(L)\ra_\RR$, then the  obvious inclusions  
$\sym_2J_\Pi\subset  \la \Pi\ra_\RR \subset \sym_2I_\Pi$ can all be strict. We denote by  $\sigma_\Pi\subset \sym_2L$ the cone spanned by $\Pi$.

Now take $L= H^g$. If  $I_\Pi$ is isotropic, then 
$\sigma_\Pi$ is contained in the parabolic  cone of $\symp(H^g_\RR)$ (via the identification of $\sym_2(H_\RR^g)$ with $\symp(H_\RR^g)$)
and the collection of such $\sigma_\Pi$ makes up a $\Sp (H^g)$-admissible decomposition  $\Sigma_\perf$ of this parabolic cone (see \cite{namikawa}). 
We thus get a toroidal  extension $\XX_g^\perf:=\XX_g^{\Sigma_\perf}$ of $\XX_g$ and a corresponding toroidal  compactification 
$\Acal_g^\perf:=\Sp (H^g)\bs \XX_g^\perf$ of $\Acal_g$.

For $\Pi$ as above,  the natural map $\XX_g\to\XX(\Pi)$ is the passage to the orbit space with respect to 
$\exp(\la \Pi\ra_\CC)$ and hence only depends on $\la \Pi\ra_\RR$. 
A point of $\XX_g$ can be understood as giving a pure polarized Hodge structure on $H^g$ of weight 1. But we can also regard it as defining 
a mixed Hodge structure on $H^g$  with weight filtration $W_{-1}=\{0\}\subset W_0= I_{\Pi,\QQ}\subset W_1= I_{\Pi,\QQ}^\perp\subset W_2=H^g_\QQ$.
In an algebro-geometric  context the image  $F_\Pi$ of $F$ in $\XX(\Pi)$ is often considered as representing a $\exp(\la \Pi\ra_\CC)$-orbit of 
mixed  Hodge structures, rather than  of pure Hodge structures. The fact that we only care about this orbit implies  that $F_\Pi$  only depends on
$F\cap J^\perp_{\Pi, \CC}$, or equivalently, on the image of $F$ in $H^g_\CC/J_{\Pi, \CC}$. So we can also view $F_\Pi$ as an orbit of
mixed  Hodge structures on $J^\perp_\Pi$ (with $W_2=J^\perp_{\Pi, \QQ}$).

The stratum $\XX(\Pi)$ has the structure of an (iterated) affine bundle over 
$\XX(I_\Pi^\perp/I_\Pi)$);  its complex codimension in $\XX_g^\perf$  is equal to $\dim\la \Pi\ra_\RR$. 

The associated category $\Tfrak^\perf_g:=\Tfrak^{\Sigma_\perf}_{\Sp (H^g)}$ has as its objects the perfect subsets $\Pi\subset \sym_2 H^g$ for which $I_\Pi$ is isotropic.
A morphism $\Pi\to \Pi'$ is given by a $\g\in \Sp (H^g)$ such that $\g(\Pi)\subset \Pi'$ with the understanding that $\g'\in \Sp (H^g)$
defines the same morphism if and only if $\g'\g^{-1}\in \exp(\la \Pi '\ra_\QQ)$. The functor $\Tfrak^\perf_{\Sp (H^g)}\to 
\Wfrak_{\Sp (H^g)}$  that incarnates the $\QQ$-homotopy class of $\Acal_g^\perf\to \Acal^\bb_g$ is given by $\Pi\mapsto I_\Pi$.

\begin{remark}
There is also an analogue $\Wfrak^\perf_g$ of Giffen's category: an object of $\Wfrak^\perf_g$ is  a pair $(L,\Pi)$, where $L$ is  a quasi-unimodular symplectic lattice (i.e., $L/\rad (L)$  is unimodular)  with $\rk (\rad (L))+\rk L=2g$ and $\Pi$ is a perfect subset of $\Sq(\rad(L))$ such that $I(\Pi)=\rad(L)$. A morphism $(L,\Pi)\to (L', \Pi')$ is given by a primitive symplectic embedding $\phi: L'\to L$ (note the contravariance) such that $\phi^{-1}\Pi\subset \Pi'$.
Then $|B\Wfrak_g^\perf|$ is $\QQ$-homotopy equivalent to $\Acal_g^\perf$. We have a functor $P_g: \Wfrak^\perf_g\to \Wfrak_g$ which sends $(L, \sigma)$ to $\bar L$ (any  $\phi: L'\to L$ as above  defines a primitive isotropic sublattice $K\subset  L/\rad (L)$  and an isomorphism  $L'/\rad (L')\xrightarrow{\cong} K^\perp/K$; the inverse of the latter defines a morphism in $\Wfrak_g$) and the resulting map $|BP_g|: |B\Wfrak^\perf_g|\to |B\Wfrak_g|$ incarnates the $\QQ$-homotopy class of $\Acal_g^\perf\to \Acal^\bb_g$.

The functor $\mathfrak{i}: \Wfrak^\perf_g\to \Wfrak^\perf_{g+1}$ defined by $(L, \Pi)\mapsto (L\oplus \ZZ^2,\Pi)$ (where $\ZZ^2$ is equipped with its 
standard symplectic form) reproduces the $\QQ$-homotopy class of $i_E: \Acal_g^\perf\to \Acal_{g+1}^\perf$ defined by multiplication with a fixed elliptic 
curve $E$. In fact, this is the restriction of  a morphism 
$\Acal_1^\perf\times\Acal_g^\perf\to \Acal_{g+1}^\perf$ which over the cusp of $\Acal_1^\perf=\Acal_1^\bb$ is the map
$i_\infty : \Acal_g^\perf\to \Acal_{g+1}^\perf$  that normalizes the boundary. So these maps are homotopic.

The map $i_E$ is transversal  to the strata and has a well-defined normal bundle of rank $g+1$ in the orbifold sense (it is a direct sum  of  the  dual of the Hodge bundle on $\Acal_g^\perf$  and  the trivial line bundle that comes from varying the elliptic curve) and so we have also defined a  Gysin map: 
\[
i_E^!: H_{2g+2+k}(\Acal_{g+1}^\perf; \QQ)\to H_{2g+2+k}(\Acal_{g+1}^\perf,\Acal_{g+1}^\perf-i_{E}(\Acal^\perf_{g}); \QQ)\cong 
H_{k}(\Acal_g^\perf, \QQ) 
\]
(a natural map followed by a Thom isomorphism). This map, which appears in the work of Grushevsky-Hulek-Tommasi \cite{ght},  is of course geometrically given as transversal pull-back along $i_E$.  But it is not clear to us whether we can phrase this in terms of the categories $\Wfrak^\perf_g$ and $\Wfrak^\perf_{g+1}$.
\end{remark}

\end{document}